\documentclass[english, 11pt]{amsart}

\usepackage{tikz}
\usepackage{aliascnt}
\usepackage{mathrsfs}
\usepackage[all,poly,knot]{xy}
\usepackage{hyperref}
\usepackage{comment}  
\usepackage{csquotes}   
\usepackage{amssymb,amsbsy,amsmath,amsfonts,amssymb,amscd,
	graphics,color,footmisc,fancyhdr,multicol,fancybox,
	graphicx,mathrsfs,rotating,ifthen,wasysym}

\usepackage{lipsum}

\usepackage{times}
\usepackage{mathtools}
\usepackage{pdfpages}
\usepackage{multirow}
\usepackage{nccrules}
\usepackage{textcomp}
\usepackage{comment}
\usepackage{framed}
\usepackage[all]{xy}
\usepackage{graphicx}
\usepackage{pgf,tikz}
\usepackage{mathrsfs}
\usetikzlibrary{arrows}
\usepackage{lipsum}
\usepackage{mathtools}

\DeclareMathOperator{\dif}{\text{\normalfont d}}

\DeclareMathOperator{\leng}{Length}

\DeclareMathOperator{\sing}{Sing}
\DeclareMathOperator{\diff}{Diff}
\DeclareMathOperator{\area}{Area}

\DeclareMathOperator{\FS}{FS}

\newcommand{\CC}{\mathbb{C}}
\newcommand{\A}{\mathbf{A}}
\newcommand{\D}{\mathbb{D}}
\newcommand{\K}{\mathsf{K}}
\newcommand{\PP}{\mathbb{P}^1}
\def\log{\mathrm{log}\,}

\theoremstyle{plain}

\newtheorem{thm}{Theorem}[section]

\newtheorem{obs}[thm]{{Observation}} 

\newtheorem{lem}[thm]{{Lemma}}

\newtheorem{pro}[thm]{Proposition}

\newtheorem{ques}[thm]{{Question}}

\theoremstyle{remark}
\newtheorem{rmk}[thm]{Remark}

\numberwithin{equation}{section}

\usepackage[top=2cm, bottom=2cm, left=2.5cm, right=2.5cm]{geometry}

\usepackage[hyperpageref]{backref}

\usepackage{xcolor}
\hypersetup{
	colorlinks,
	linkcolor={red!50!black},
	citecolor={blue!62!black},
	urlcolor={blue!80!black}
}
\theoremstyle{plain}
\newcommand{\thistheoremname}{}
\newtheorem*{genericthm*}{\thistheoremname}
\newenvironment{namedthm*}[1]{\renewcommand{\thistheoremname}{#1}%
	\begin{genericthm*}}
	{\end{genericthm*}}

\newtheoremstyle{named}{}{}{\itshape}{}{\bfseries}{.}{.5em}{\thmnote{#3's }#1}
\theoremstyle{named}

\makeatletter
\newcommand\thankssymb[1]{\textsuperscript{\@fnsymbol{#1}}}
\makeatother
\begin{document} 
	\title[On Ahlfors currents]{On Ahlfors currents}

	
	\author{Dinh Tuan Huynh}
	
	\address{Hua Loo-Keng center for Mathematical Sciences, Academy of Mathematics and System Science, Chinese Academy of Sciences, Beijing 100190, China \& Department of Mathematics, University of Education, Hue University, 34 Le Loi St., Hue City, Vietnam}
	\email{dinhtuanhuynh@hueuni.edu.vn}

	\author{Song-Yan Xie }
	\address{Academy of Mathematics and System Science \& Hua Loo-Keng Key Laboratory
		of Mathematics, Chinese Academy of Sciences, Beijing 100190, China}
	\email{xiesongyan@amss.ac.cn}

\begin{abstract}
	We answer a basic question in Nevanlinna theory that Ahlfors  currents associated to the same entire curve may be {\em nonunique}. 
Indeed, we will construct  one
 exotic entire curve $f: \mathbb{C}\rightarrow X$ which produces infinitely many cohomologically different Ahlfors currents. Moreover, concerning Siu's decomposition, for 
 an arbitrary $k\in \mathbb{Z}_{+}\cup \{\infty\}$, some of the obtained Ahlfors currents have singular parts supported on $k$ irreducible curves.
 In addition, they can have {\em nonzero} diffuse parts as well.
Lastly, we provide new examples of diffuse Ahlfors currents on the product of two elliptic curves and on $\mathbb{P}^2(\CC)$, and we show  cohomologically elaborate Ahlfors currents on blow-ups of $X$.

\end{abstract}

\maketitle

\section{\bf Introduction}

Let $X$ be a compact complex  manifold equipped with an area form $\omega$. Let $f:\mathbb{C}\longrightarrow X$ be a
nonconstant entire holomorphic curve. An associated {\sl Ahlfors current} of $f$ is a positive closed current of bidimension $(1,1)$ obtained as the weak limit of a certain  sequence of positive currents of bounded masses

\begin{equation*}
\label{radius sequence}
\bigg\{
\dfrac{[f(\mathbb{D}_{r_n})]}{\area_{\omega}f(\mathbb{D}_{r_n})}
\bigg\}_{n\geqslant 1},
\end{equation*}
where $\mathbb{D}_{r_n}$ are discs of increasing radii $r_n\nearrow \infty$ centered at the origin. Here, to ensure that such a limit current is closed, the sequence $\{r_n\}$ is chosen in such a way that the lengths of boundaries of the  discs are asymptotically negligible compared with their areas, namely
\[
\lim_{n\rightarrow\infty}
\dfrac{\leng_{\omega}(f(\partial\mathbb{D}_{r_n}))}{\area_{\omega}(f(\mathbb{D}_{r_n}))}
=
0.
\]
By Ahlfors' lemma (c.f. \cite{brunella1999, Nevanlinna1970}),
for each positive number $\epsilon>0$, the set
\[
\bigg\{
r> 0:\,\dfrac{\leng_{\omega}(f(\partial\mathbb{D}_{r}))}{\area_{\omega}(f(\mathbb{D}_{r}))}
\geqslant \epsilon
\bigg\}
\]
is of finite  measure with respect to $\frac{\dif r}{r}$. Hence the above {\it ``length-area'' condition} is satisfied for most choices of increasing radii. Moreover, given a sequence  of radii $\{r_n\}_{n\geqslant 1}$ with $r_n\nearrow\infty$, 
after some small perturbation by scaling and extracting a subsequence,
one can always obtain an Ahlfors current for $f$.

Ahlfors currents and their analogs obtained by taking the logarithmic average $\int\frac{
\dif t}{t}(\cdot)$, called {\sl Nevanlinna currents}, are fundamental tools in studying  complex hyperbolicity, value distribution theory and complex dynamical systems. Notably, they played  a crucial role in the work McQuillan \cite{Mcquillan1998} on Green-Griffiths'
conjecture for algebraic surfaces of general type having positive Segre class (see also~\cite{brunella1999} for a simplified proof by Brunella). By employing Ahlfors currents,
Duval \cite{Duval2008} gave a quantitative version of the classical Brody's Lemma and obtained a characterization of complex hyperbolicity in terms of linear isoperimetric inequality for holomorphic discs. Using such currents, some geometric refinement of the classical Cartan's Second Main Theorem \cite{Duval-Huynh2018}, as well as the high dimensional Weierstrass-Casorati Theorem \cite{Huynh-Vu2020} were obtained. The reader is also referred to \cite{Dinh-Sibony2018} for recent key applications in complex dynamical systems. 

Since Ahlfors currents and Nevanlinna currents encode geometric information of their original entire curves, several results in value distribution theory can be presented in terms of intersections of corresponding cohomology classes. For example, the First Main Theorem of Nevanlinna theory can be expressed as an inequality between the algebraic intersection and the geometric intersection (c.f.~\cite{Duval-Huynh2018}).

 Note that in certain specific situations, Ahlfors currents (or Nevanlinna currents) from some holomorphic curve are unique \cite{Dinh-Sibony2014, Duval-Huynh2018, Dinh-Sibony2018, Dinh-Vu2020}, which subsequently leads to several interesting results.
Therefore, it is natural and  fundamental to ask generally

\begin{ques}
	\label{question 1}
	Are all Ahlfors currents associated to the same entire curve cohomologically equivalent?
\end{ques}

The study of such currents  is itself of independent interest. By Siu's decomposition Theorem~\cite{Siu1974}, an Ahlfors current $T$ can be written as the sum $T=T_{\sing}+T_{\diff}$, where the singular part $T_{\sing}=\sum_{\ell\in I}c_{\ell}\cdot[C_\ell]$ is some positive linear combination ($c_\ell> 0$; $I\subset \mathbb{Z}_+$, could be $\varnothing$) of
currents of integration on irreducible algebraic curves $C_\ell$, and where the diffuse part $T_{\diff}$ is a positive closed $(1,1)$--current having zero Lelong number along any algebraic curve. If the singular part $T_{\sing}$ is nontrivial, Duval~\cite{Duval2006} showed that  any irreducible curve $C_\ell$ above must be rational or elliptic (see also~\cite{Duval2017} for a local version). In \cite{dacosta2013}, da Costa gave an example of entire curve in the projective plane whose associated Ahlfors current is supported in some line. This construction can be modified to produce Ahlfors currents  supported on a rational or an elliptic curve \cite[Theorem 2.6.1]{Huynh2016}. On the other hand, we would like to mention the following unsolved question, which had been considered  by
Brunella~\cite[page~200]{brunella1999}. 
\begin{ques}
Is there any Ahlfors current from an entire curve such that both of its singular part and diffuse part are nontrivial?
\end{ques}

In this paper, we answer the above two questions by constructing explicit examples.

	\begin{thm}
		\label{thm 1}
		There exists an entire curve producing 
		cohomologically different
		Ahlfors currents.
	\end{thm}

 By Siu's decomposition, Ahlfors currents with nontrivial singular parts can be distinguished as different types by the data
 $(|I|\in \mathbb{Z}_+\cup \{\infty\}, T_{\diff} \text{ is trivial / nontrivial})$.
 
\begin{thm}
	\label{thm 2}
	 There exists an entire curve producing all types of Ahlfors currents with nontrivial singular parts. 
	\end{thm}

\begin{rmk}
	The above two results also hold true for Nevanlinna currents, see Subsection~\ref{Singular Nevanlinna currents on X}.
\end{rmk}

Lastly,
it is natural to seek Ahlfors (Nevanlinna) currents with trivial singular part. Examples of such currents
are known to exist on $\mathbb{P}^2(\mathbb{C})$,
 by looking at the Levi-flat real hypersurface
 in $(\mathbb{C}^*)^2$ defined by the equation $|x|=|y|^{\alpha}$, where $\alpha$ is an irrational real number (c.f.~\cite[page~262]{Duval-Berteloot2001}). Indeed, this real hypersurface
 is foliated by entire curves, while its closure in $\mathbb{P}^2(\mathbb{C})$ contains no algebraic curve.
 In \cite{Dinh-Sibony2014}
 there are more examples of  holomorphic curves 
  whose associated Ahlfors  (Nevanlinna) currents are  diffuse and unique. In Section~\ref{section: examples}, we show new examples of diffuse Ahlfors currents on the product of two elliptic curves and on $\mathbb{P}^2(\mathbb{C})$, see Propositions~\ref{diffuse in Abelian surface}, \ref{diffuse Ahlfors currents in Cp2}.

\bigskip

We now outline the ideas and the structure of this paper. As a matter of fact, our source of inspiration is an example of da Costa  \cite{dacosta2013} about a nondegenerate entire curve clustering to a line in $\mathbb{P}^2(\mathbb{C})$ (see also~\cite{Duval-Huynh2018} for more discussions). In Section \ref{construction}, we start with an elliptic curve $\mathcal{C}=\mathbb{C}/\Gamma$ equipped with a negative line bundle $\mathcal{L}$. For some large integer $m\gg 1$, we   construct a section $s_m$ of $\pi_0^*\mathcal{L}^m$ having large exponential growth of order $2$, where $\pi_0:\mathbb{C}\rightarrow\mathcal{C}$ is the canonical projection. The surface $X$ is obtained by taking the geometric projectivization  $\mathbb{P}(\mathcal{L}^m\oplus\mathbb{C})=:X$ of the vector bundle $\mathcal{L}^m\oplus\mathbb{C}$ on $\mathcal{C}$. Thus the section $s_m$ induces a holomorphic map $f_0: \mathbb{C}\rightarrow X$ clustering to the curve $\mathcal{C}_{\infty}$ corresponding to the ``infinity section'' of $\pi_0^*\mathcal{L}^m$.  To generate Ahlfors currents with larger singular supports, we hence modify the original section $s_m$ by multiplying it with a  Weierstrass canonical product
	$
\psi(z)=\prod_{\lambda\in\Lambda}\Big(1-\frac{z}{\lambda}\Big)e^{\frac{z}{\lambda}+\frac{z^2}{2\lambda^2}}
$, whose zero locus $\Lambda$ is distributed in a delicate pattern,
to make sure that the new section $\psi\cdot s_m$   induces an entire curve $f:\mathbb{C}\longrightarrow X$ producing  Ahlfors currents with more singularities. Indeed, for every $\lambda\in\Lambda$, since $\psi\cdot s_m(\lambda)=0$,
$f(\lambda)$ touches the curve $\mathcal{C}_0:=\mathcal{C}\times [0 \oplus 1]$, defined by the zero section of  $\pi_0^*\mathcal{L}^m$, at $([\lambda], [0 \oplus 1])$.

 The idea is that, the image of a small neighborhood of $\lambda\in \Lambda\subset \mathbb{C}$ by $f$ shall contribute moderate area $O(1)$ near the fiber $\mathbb{P}^1_{[\lambda]}\subset X$ over $[\lambda]\in \mathcal{C}$, and once  there are sufficiently many $\lambda'\in \Lambda$ mapping to the same class $[\lambda]$ by $\pi_0$,
the area of the image of $f$ should spend a positive portion about $\mathbb{P}^1_{[\lambda]}$, hence the Ahlfors currents should charge positive mass there. See the picture below for illustration.

\begin{center}
	\scalebox{.85}{\begin{picture}(0,0)%
\includegraphics{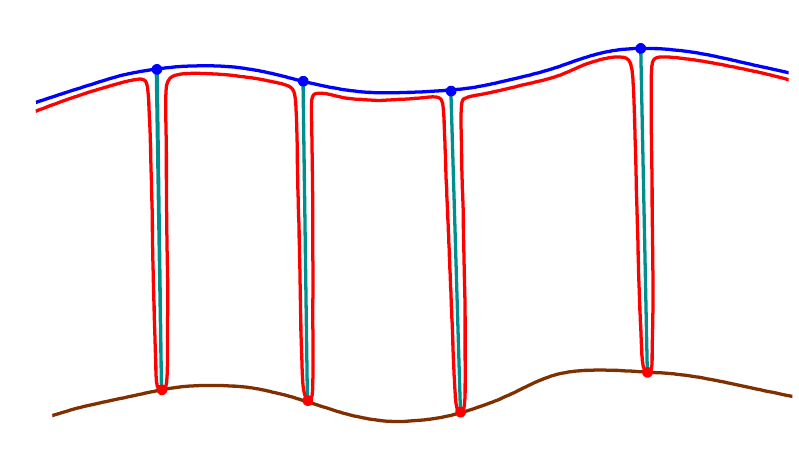}%
\end{picture}%
\setlength{\unitlength}{4144sp}%
\begingroup\makeatletter\ifx\SetFigFont\undefined%
\gdef\SetFigFont#1#2#3#4#5{%
  \reset@font\fontsize{#1}{#2pt}%
  \fontfamily{#3}\fontseries{#4}\fontshape{#5}%
  \selectfont}%
\fi\endgroup%
\begin{picture}(3644,2081)(1205,-2501)
\put(1283,-1059){\makebox(0,0)[lb]{\smash{{\SetFigFont{7}{8.4}{\rmdefault}{\mddefault}{\updefault}{\color[rgb]{1,0,0}$f(\mathbb{C})$}%
}}}}
\put(4225,-1463){\makebox(0,0)[lb]{\smash{{\SetFigFont{7}{8.4}{\familydefault}{\mddefault}{\updefault}{\color[rgb]{0,.56,.56}$\cdots$}%
}}}}
\put(1220,-783){\makebox(0,0)[lb]{\smash{{\SetFigFont{7}{8.4}{\rmdefault}{\mddefault}{\updefault}{\color[rgb]{0,0,1}$\mathcal{C}_{\infty}$}%
}}}}
\put(4096,-2247){\makebox(0,0)[lb]{\smash{{\SetFigFont{7}{8.4}{\rmdefault}{\mddefault}{\updefault}{\color[rgb]{1,0,0}$\cdots$}%
}}}}
\put(4070,-543){\makebox(0,0)[lb]{\smash{{\SetFigFont{7}{8.4}{\rmdefault}{\mddefault}{\updefault}{\color[rgb]{0,0,1}$\cdots$}%
}}}}
\put(1788,-2331){\makebox(0,0)[lb]{\smash{{\SetFigFont{7}{8.4}{\rmdefault}{\mddefault}{\updefault}{\color[rgb]{1,0,0}$([\lambda_1],0)$}%
}}}}
\put(3179,-2432){\makebox(0,0)[lb]{\smash{{\SetFigFont{7}{8.4}{\rmdefault}{\mddefault}{\updefault}{\color[rgb]{1,0,0}$([\lambda_3],0)$}%
}}}}
\put(1726,-651){\makebox(0,0)[lb]{\smash{{\SetFigFont{7}{8.4}{\rmdefault}{\mddefault}{\updefault}{\color[rgb]{0,0,1}$([\lambda_1],\infty)$}%
}}}}
\put(2408,-676){\makebox(0,0)[lb]{\smash{{\SetFigFont{7}{8.4}{\familydefault}{\mddefault}{\updefault}{\color[rgb]{0,0,1}$([\lambda_2],\infty)$}%
}}}}
\put(3043,-718){\makebox(0,0)[lb]{\smash{{\SetFigFont{7}{8.4}{\rmdefault}{\mddefault}{\updefault}{\color[rgb]{0,0,1}$([\lambda_3],\infty)$}%
}}}}
\put(1409,-2413){\makebox(0,0)[lb]{\smash{{\SetFigFont{7}{8.4}{\rmdefault}{\mddefault}{\updefault}{\color[rgb]{.5,.17,0}$\mathcal{C}_{0}$}%
}}}}
\put(2392,-2437){\makebox(0,0)[lb]{\smash{{\SetFigFont{7}{8.4}{\familydefault}{\mddefault}{\updefault}{\color[rgb]{1,0,0}$([\lambda_2],0)$}%
}}}}
\put(2000,-1502){\makebox(0,0)[lb]{\smash{{\SetFigFont{7}{8.4}{\familydefault}{\mddefault}{\updefault}{\color[rgb]{0,.56,.56}$\mathbb{P}^1_{[\lambda_1]}$}%
}}}}
\put(2667,-1515){\makebox(0,0)[lb]{\smash{{\SetFigFont{7}{8.4}{\familydefault}{\mddefault}{\updefault}{\color[rgb]{0,.56,.56}$\mathbb{P}^1_{[\lambda_2]}$}%
}}}}
\put(3362,-1498){\makebox(0,0)[lb]{\smash{{\SetFigFont{7}{8.4}{\familydefault}{\mddefault}{\updefault}{\color[rgb]{0,.56,.56}$\mathbb{P}^1_{[\lambda_3]}$}%
}}}}
\end{picture}%
}
\end{center}

Nevertheless, to make sense of this idea, we need to show, first of all, that  the growth of $\psi$ is neither too rapid nor too slow, which will be accomplished  in Section \ref{section prep}, 
by means of the Stirling formula as well as the symmetry of  the lattice $\Gamma$. Consequently, in Section~\ref{section: estimate area of discs}, we can manipulate Jensen's formula 
to evaluate various  areas, which distinguish the singularities of  the  Ahlfors currents.
  In Section~\ref{section: algorithm}, we present an algorithm for constructing the zero locus $\Lambda$, which is designed for the proofs of the main theorems in Section~\ref{section: proof}.
In Section~\ref{section: examples}, we provide new examples of diffuse Ahlfors currents. Moreover, we show
cohomologically elaborate  Ahlfors currents
on surfaces obtained by blowing-up $X$.

\bigskip\noindent
{\sl Convention:} Throughout this paper,  $\K$  denotes positive numbers which are uniformly bounded from both sides 
$0<K_1<\K<K_2<\infty$.
Further, notation $\K_{\star_1, \star_2, \star_3}$
indicates dependence on parameters $\star_1, \star_2, \star_3$. The notation
$\D(a, r)$ means the disc centered at $a$, in $ \CC$ or in the elliptic curve $ \CC/\Gamma$,
 with the radius $r$. When $a=0\in \mathbb{C}$, we write
  $\D_{r}$ instead of $\D(0, r)$. The differential operator $\dif^c$ is short for $\frac{\sqrt{-1}}{4\pi}(\overline{\partial}-\partial)$.

\vspace*{0.5cm}
\noindent{\bf Acknowledgements.} We would like to address our profound gratitude to Professor Julien Duval for introducing us to the subject and for many fruitful discussions from which we learned rich ideas.
We are grateful to Professor Sibony for sharing the reference~\cite{Dinh-Sibony2014} and for suggesting related problems. We  thank Professor Yusaku Tiba for his question about cohomology classes of Ahlfors currents. We 
thank  Professor Jo\"el Merker, Dr. Duc-Viet Vu and Dr. Ruiran Sun for valuable comments and suggestions. 
S.-Y. Xie is partially supported by  the NSFC Grant No.~11688101.
Part of this paper was written during his visit to Tianyuan Mathematical Center in Southeast China / Xiamen University invited by Professor Chunhui Qiu. He is  grateful to the colleagues there for hospitality and excellent working conditions. D. T. Huynh is
grateful to Academy of Mathematics and Systems Science in Beijing for enhanced scientific ambience. He also wants to acknowledge partial support from the Core Research Program of Hue University, Grant No. NCM.DHH.2020.15.

\section{\bf Construction}
\label{construction}
Fix 
a smooth elliptic curve $\mathcal{C}=\mathbb{C}/\Gamma$, where  the lattice
 $\Gamma:=\mathbb{Z}
\oplus \mathbb{Z}\sqrt{-1}$
is chosen  to affiliate the arguments later.
 We can find a negative line bundle $\mathcal{L}$ on $\mathcal{C}$ equipped with some  hermitian metric $h'$ having strictly negative curvature. Now comparing with the K\"ahler form $\dif\dif^c\big(|z|^2\big)$ on $\mathcal{C}$ descending from the canonical projection $\pi_0:\mathbb{C}\longrightarrow \mathbb{C}/\Gamma$, 
 the curvature of $h'$ is cohomologous to $-2\alpha \dif\dif^c\big(|z|^2\big)$
for some positive constant $\alpha$, namely their difference is of the form $\dif\dif^c\varphi$ for some smooth real function $\varphi$ on $\mathbb{C}/\Gamma$. Therefore, replacing the initial metric $h'$ by $h'e^{\varphi}=:h$, the curvature becomes
$
\Theta_h
=
-2\alpha\, \dif\dif^c\big(|z|^2\big)$.
Noting that the line bundle
$\pi_0^*\mathcal{L}$ on $\mathbb{C} $ is holomorphically trivial,
it has a nowhere vanishing holomorphic section $k$, which by Lelong-Poincar\'e equation satisfies that
$
\dif\dif^c\big(\log\|k\|^2_{\pi_0^*h}\big)
=
2\alpha\,\dif\dif^c\big(|z|^2\big)
$.
Hence $\log\|k\|^2_{\pi_0^*h}-2\alpha\,|z|^2$ is a harmonic function on $\mathbb{C}$, hence can be written as the real part of some holomorphic function $g$. Therefore, the  modified  section
	 $s:=e^{-g/2}\, k$
	 of $\pi_0^*\mathcal{L}$ has exponential growth of order two
	 $\|s\|_{\pi_0^*h}
	 \,
	 =\,
	 \|e^{-g/2}\, k\|_{\pi_0^*h}
	 \,
	 =
	 \,
	 e^{\alpha |z|^2}$. 
	 The above construction is based on an idea of~\cite{dacosta2013}.
	 
	 We now amplify the negativity of $\mathcal{L}$, by
	 introducing
	 $\mathcal{L}_m:=\mathcal{L}^{\otimes m}$ for some big multiplicity $m\geqslant 1$ to be determined. For the metric  $h_m:=h^{\otimes m}$ of $\mathcal{L}_m$, the section $s_m:=s^{\otimes m}$  has large exponential growth
	\begin{equation}
	\label{exponential growth of s}
	\|s_m\|_{h_m}
	\,
	=
	\,
	e^{m\alpha |z|^2}.
	\end{equation}

	Now we introduce the complex surface $X:=\mathbb{P}(\mathcal{L}_m\oplus \mathbb{C})$ obtained by the geometric projectivization of the rank $2$ vector bundle $\mathcal{L}_m\oplus \mathbb{C}$
	over $\mathcal{C}$. Denote by $\pi_1: X\longrightarrow \mathcal{C}$ the canonical projection.
	By 
	 the fiberwised identification $\mathcal{L}_m\cong \mathcal{L}_m\oplus 1\subset \mathbb{P}(\mathcal{L}_m\oplus \mathbb{C})$,
	 the tautological space of $\mathcal{L}_m$ can be embedded into $X
	$ as an open subset, whose complement is the elliptic curve $\mathcal{C}_{\infty}:=\mathcal{C}\times [1 \oplus 0]\subset \mathbb{P}(\mathcal{L}_m\oplus \mathbb{C})$ at ``infinity''.

Next, we introduce an auxiliary holomorphic function 
	\[
	\psi(z):=\prod_{\lambda\in\Lambda}\Big(1-\frac{z}{\lambda}\Big)e^{\frac{z}{\lambda}+\frac{z^2}{2\lambda^2}}
	\]
	obtained by Weierstrass canonical product,
	 where the zero locus $\Lambda$ will be chosen carefully by the following sophisticated reasoning, to make sure that the global section $\psi\cdot s_m$ 
	 of $\pi_0^*\mathcal{L}_m$ together with the inclusion $\iota: \mathcal{L}_m\hookrightarrow X$ 
	 induce an entire curve $f: \mathbb{C}\longrightarrow X$ producing complicated Ahlfors currents.

 First of all, we  would like to have the estimate $\log |\psi(z)|\leqslant O(|z|^2)$,
	 at least for $|z|$ around $r_i$ for some specific radii $r_i\nearrow \infty$, in order to bound the area of ${f}(\mathbb{D}_{r_i})$
	by $O(r_i^2)$. 
	 
	 Secondly, we require that the cardinality $ |\Lambda\cap \mathbb{D}_{r_i}|=O(r_i^2)$, so that the image  ${f}(\mathbb{D}_{r_i})$ intersects the curve $\mathcal{C}_0:=\mathcal{C}\times [0 \oplus 1]\subset X$ defined by the zero section of $\mathcal{L}_m$ frequently enough. 
	 
	 Lastly, we require that each time when the image of the entire curve $f$ intersects $\mathcal{C}_0$ for  $\lambda\in \Lambda$ with $|\lambda|\gg1$, it  contributes  $O(1)$ area near the fiber $\mathbb{P}^1_{[\lambda]}:=\pi_1^{-1}([\lambda])$.

	\smallskip  
Thus we declare that
\begin{enumerate}
\item[(i)]
near each  annulus $\mathbf{A}_{r_i}:=\{z\in\mathbb{C}: \frac{r_2}{2}\leqslant |z|\leqslant r_i\}$, the zero locus $\Lambda$ is a mild perturbation of $ \mathbf{A}_{r_i}\cap c\,\Gamma$, 	where $c\geqslant 5$ is some positive integer to be determined. More precisely
\[
\Lambda=\cup_{i
\geqslant  1}
	\,
B_{r_i}
\qquad
\text{where}\qquad
{B}_{r_i}
:=
\cup_{\mu\in \mathbf{A}_{r_i}\cap c\Gamma}\,
\{\mu+x_{\mu}\}.
\]
Here at the moment we  only tell that all $x_{\mu}$'s	take values in the fundamental domain 
\[
 \mathcal{D}:=\{x+y\sqrt{-1}\,:\,
0\leqslant x, y<1
\},
\]
and later in Section~\ref{section: algorithm},	we will elaborate on the choices of $x_{\mu}$'s for delicate reasons.
\smallskip
	 	
\item[(ii)] 
$\{r_i\}_{i\geqslant 1}$ grow very rapidly, say 
\begin{equation}
	\label{r_i grow rapid}
r_1\geqslant 2020\cdot c,
\qquad
 {r_{i+1}}\geqslant {r_i}^4 
 \quad
 {\scriptstyle
 	(\forall\, i\,\geqslant\,1)
 }.
\end{equation}
\end{enumerate}

\section{\bf Preparations}	 
\label{section prep}
\begin{lem}
	\label{lemma 1}
	One has a uniform estimate
	$
	\vert
	\sum_{\lambda\in B_{r_i}}\,
	\frac{1}{\lambda}
	\vert 
	\leqslant
	\K/{c^2}
	$
	for all $i=1, 2, \dots$.
	\end{lem}

\begin{proof}
	In the special case that all $x_{\mu}=0$, by the
	 symmetry of $\Gamma$ that $(-1)\cdot \Gamma=\Gamma$
	 and that of $\mathbf{A}_r$,
	 the sum 
	 $\sum_{\mu\in \mathbf{A}_{r_i}\cap c\Gamma}\,\frac{1}{\mu+0}$
	 is always $0$. 

	 In general,
	  for every $\mu\in \mathbf{A}_{r_i}
	 \cap c\Gamma$, one has the estimate $|\frac{1}{\mu+x_{\mu}}-\frac{1}{\mu+0}|\leqslant \K/r_i^2$.
	 Noting that the cardinality $ |B_{r_i}|\leqslant
	 \K\cdot (r_i/c)^2$, we conclude that
	 $\vert
	 \sum_{\lambda\in B_{r_i}}\,
	 \frac{1}{\lambda}
	 \vert
	 \leqslant
	 \K\cdot (r_i/c)^2
	 \cdot
	 \K/r_i^2
	 =
	 \K
	 /
	 c^2
	 $.
	 \end{proof}

\begin{lem}
	\label{lemma 2}
	One has a uniform estimate
	$
	\vert
	\sum_{\lambda\in B_{r_i}}\,
	\frac{1}{\lambda^2}
	\vert 
	\leqslant
	\K/{(c^2\,r_i)}
	$
	 for all $i=1, 2, \dots$.
\end{lem}

\begin{proof}
The argument goes much the same way as the preceding one, by using the rotational symmetry of $\Gamma$ that
	$\sqrt{-1}\cdot\Gamma=\Gamma$.
	Indeed, we  have the identity 
	$\sum_{\mu\in \mathbf{A}_{r_i}\cap c\Gamma}\,\frac{1}{
		(\mu+0)^2}=0$. Moreover, for every $\mu\in \mathbf{A}_{r_i}
	\cap c\Gamma$, 
	we have 
	$|\frac{1}{
		(\mu+x_{\mu})^2}-\frac{1}{
		(\mu+0)^2}|
\leqslant 
\K/r_i^3
	$.
	The remaining argument is clear.
\end{proof}

We make a convention that $\log 0=-\infty$.

\begin{pro}
	For every $i\geqslant 2$
	and for $r_i/3\leqslant |z|\leqslant 3r_i$, one has
	\begin{equation}
		\label{psi<}
	\log |\psi(z)|\leqslant \K
 \cdot r_i^2/c^2.
 \end{equation}
	\end{pro}

To bound the area  of $f(\D_{r_i})$ by $\K\cdot r_i^2$, it is crucial to have the above estimate. In fact, by classical complex analysis (c.f.~\cite[Chapter 4]{Levin1996}), we can check that $\psi$
is well-defined and that the infinite product is uniformly convergent in bounded domains, and that the exponential growth order of $\psi$ is, by applying  Borel's formula~\cite[page~30, Theorem 3]{Levin1996}, exactly $2$, {\em i.e.} 
$\log |\psi(z)|\leqslant \K_{\epsilon}\cdot |z|^{2+\epsilon}$ for  any $\epsilon>0$ and for large $|z|$.
Nevertheless, for the critical case that $\epsilon=0$ we need more effort.

\begin{proof}
	We first study $I:=\prod_{\ell=1}^{i-1}\prod_{\lambda\in B_{r_\ell}}(1-\frac{z}{\lambda})e^{\frac{z}{\lambda}+\frac{z^2}{2\lambda^2}}$ concerning  smaller annuli compared with  $\A_{r_i}$. Note that
	$\prod_{\ell=1}^{i-1}\prod_{\lambda\in B_{r_\ell}}
	|1-\frac{z}{\lambda}|\leqslant (1+3r_i)^{ |B_{r_1}|+\cdots +| B_{r_{i-1}}|}\leqslant
	(1+3r_i)^{\K\cdot r_i/c^2}
	$. Hence by Lemmas~\ref{lemma 1},~\ref
	{lemma 2}, we receive that
	$\log |I|\leqslant \log (1+3r_i)^{\K\cdot r_i/c^2}
	+ \sum_{\ell=1}^{i-1}
	(|\sum_{\lambda\in B_{r_\ell}} \frac{1}{\lambda}||z|
	+
	|\sum_{\lambda\in B_{r_\ell}} \frac{1}{\lambda^2}||\frac{z^2}{2}|)
	\leqslant
	\K\cdot r_i^2/c^2$.
	
	Secondly, we observe $II:=\prod_{\lambda\in B_{r_i}}(1-\frac{z}{\lambda})e^{\frac{z}{\lambda}+\frac{z^2}{3\lambda^2}}$ concerning the annulus  $\A_{r_i}$. Note that each term $|1-\frac{z}{\lambda}|\leqslant \K$ by our construction, and that $| B_{r_i}|\leqslant \K\cdot (r_i/c)^2$. Now using Lemmas~\ref{lemma 1},~\ref
	{lemma 2}, we receive that
	$\log |II|\leqslant \K\cdot r_i^2/c^2$.
	
	Lastly, we analyze $III:=\prod_{\ell\geqslant i+1}\prod_{\lambda\in B_{r_\ell}}(1-\frac{z}{\lambda})e^{\frac{z}{\lambda}+\frac{z^2}{2\lambda^2}}$ concerning  larger annuli compared with  $\A_{r_i}$.
	Now the key point is that, for each
	$\lambda\in B_{r_\ell}$, one has $|\frac{z}{\lambda}|\leqslant \frac{3r_i}{r_{\ell}/3}\ll 1$. Therefore we can apply the Taylor expansion of
	$\log (1-\frac{z}{\lambda})$ to
	achieve desired estimates. Indeed, noting that
	$
	\log\big( (1-\frac{z}{\lambda})e^{\frac{z}{\lambda}+\frac{z^2}{2\lambda^2}}\big)
	=
	-\sum_{n\geqslant 3}
	\frac{1}{n}
	(\frac{z}{\lambda})^n$,
	 hence
	 \[
	 	 \Big|
	 \log
	 \prod_{\lambda\in B_{r_\ell}}(1-\frac{z}{\lambda})e^{\frac{z}{\lambda}+\frac{z^2}{2\lambda^2}}
	 \Big|
	 \leqslant
	 \sum_{\lambda\in B_{r_\ell}}
	 \sum_{n\geqslant 3}
	 \frac{1}{n}
	 \Big|
	 \frac{z}{\lambda}
	 \Big|^n
	 \leqslant
	 \K\cdot r_\ell^2/c^2
	 \cdot
	 \sum_{n\geqslant 3}
	 \frac{1}{n}
	 \Big|
	 \frac{3r_i}{r_{\ell}/3}
	 \Big|^n
	 \leqslant
	 \K/c^2
	 \cdot
	 \frac{r_i^3}{r_\ell}
	 \cdot \K.
	 \]
	 Since the sequence $\{r_{\ell}\}_{\ell>i}$ grows very rapid by our construction~\eqref{r_i grow rapid}, there holds
	 $\sum_{\ell>i} \frac{r_i^3}{r_\ell}
	 <\K
	 $. Thus the above estimate yields $|\log III|\leqslant \K/c^2$.
	
Combining all the above  estimates about $I, II, III$, we receive the desired inequality~\eqref{psi<}.
\end{proof}

By our construction of $\Lambda$, it intersects each  disc $\D(z, 1)$ at most once.
Therefore we introduce
	\begin{equation}
		\label{psi_1}
\psi_1(z):=\prod_{\lambda\in\Lambda\setminus \D(z, 1)}\Big(1-\frac{z}{\lambda}\Big)e^{\frac{z}{\lambda}+\frac{z^2}{2\lambda^2}}
\end{equation}
to capture the asymptotic behavior of $\psi$ away from its zero locus $\Lambda$.

\begin{pro}
	\label{the most difficult estimate of psi}
	For every  $z\in \CC$ with  large $|z|$, 
	one has
	\begin{equation}
		\label{psi>}
	\log |\psi_1(z)
	|
	\geqslant -\K\cdot |z|^2/c^2.
\end{equation}
\end{pro}

\begin{proof}
	Fix a  positive small number $\eta=\frac{1}{100}$. For every $i\geqslant 1$,
	we introduce the slightly larger annulus 
	$\widetilde{\A}_{r_i}:=\{x\in \CC : (1-\eta )r_i/2\leqslant|x|
	\leqslant (1+\eta)r_i\}\supseteq \A_{r_i}$, to make sure that $\widetilde{\A}_{r_i}\supseteq B_i$.
	
	Case $(i)$: $|z|$ large with $z\notin \cup_{i
\geqslant} \widetilde{\A}_{r_i}$. Then  $z$ lies between some two consequent annuli $\widetilde{\A}_{r_j}$
and $\widetilde{\A}_{r_{j+1}}$, i.e.,
$(1+\eta)r_j< |z|< (1-\eta)r_{j+1}/2$, and it is clear that $\psi(z)=\psi_1(z)$. Firstly, 
for each $\lambda\in \cup_{\ell=1}^j B_{r_\ell}$,
we have 
$|1-\frac{z}{\lambda}|\geqslant |\frac{z}{\lambda}|-1\geqslant \eta':=\eta/2$.
Thus
$\prod_{\ell=1}^{j}\prod_{\lambda\in B_{r_\ell}}
|1-\frac{z}{\lambda}|\geqslant \eta'^{\sum_{\ell=1}^j |B_{r_{\ell}}|}\geqslant
\eta'^{\K\cdot r_j^2/c^2}
\geqslant \eta'^{\K\cdot |z|^2/c^2}$. Next, thanks to
Lemmas~\ref{lemma 1},~\ref
{lemma 2}, we have
$\prod_{\ell=1}^{j}\prod_{\lambda\in B_{r_\ell}}
|e^{\frac{z}{\lambda}+\frac{z^2}{2\lambda^2}}|
\geqslant
e^{-|z|\sum_{\ell=1}^j\K/c^2-|z|^2\sum_{\ell=1}^j\K\cdot /(c^2 r_{\ell})}\geqslant
e^{-\K\cdot |z|^2/c^2}$.
Lastly, by mimicking the estimate of $III$ in the preceding proof, we receive that
\[
\log
|
\prod_{\ell\geqslant j+1}\prod_{\lambda\in B_{r_\ell}}(1-\frac{z}{\lambda})e^{\frac{z}{\lambda}+\frac{z^2}{2\lambda^2}}
|
\geqslant 
-\sum_{\ell\geqslant j+1}\sum_{\lambda\in B_{r_\ell}}
	\sum_{n\geqslant 3}
\frac{1}{n}
|\frac{z}{\lambda}|^n
\geqslant
-\K\cdot
|z|^2/c^2.
\]
Summarizing the above estimates, 
we conclude that
$\log |\psi_1(z)
|
=
\log |\psi(z)|
\geqslant -\K\cdot |z|^2/c^2$.
	
	Case $(ii)$: $|z|$ large with $ z\in\widetilde{\A}_{r_j}$ for some $j$.
	By repeating the same arguments as above, we can show that, first of all,
	$
	\log
	|
	\prod_{\ell=1}^{j-1}\prod_{\lambda\in B_{r_\ell}}(1-\frac{z}{\lambda})e^{\frac{z}{\lambda}+\frac{z^2}{2\lambda^2}}
	|
	\geqslant
	-\K\cdot |z|^2/c^2$,
	and secondly, 
	$
	\log
	|
	\prod_{\ell\geqslant j+1}\prod_{\lambda\in B_{r_\ell}}(1-\frac{z}{\lambda})e^{\frac{z}{\lambda}+\frac{z^2}{2\lambda^2}}
	|
	\geqslant
	 -\K\cdot |z|^2/c^2$.
	 By
	 Lemmas~\ref{lemma 1},~\ref
	 {lemma 2}, we  receive
	 $
	 \log
	 |
	 \prod_{\lambda\in B_{r_j}\setminus \D(z, 1)}e^{\frac{z}{\lambda}+\frac{z^2}{2\lambda^2}}
	 |
	 \geqslant
	 -\K\cdot |z|^2/c^2$.
	 Thus the remaining problem   is to show that
	$
	\log
	|
	\prod_{\lambda\in B_{r_j}\setminus \D(z, 1)}(1-\frac{z}{\lambda})
	|
	\geqslant
	-\K\cdot |z|^2/c^2$. 
	
	To start with, we find a point $\mu_0$ in $c\Gamma$ having the least Euclidean distance to $z$.
	Then for every $\lambda=\mu+x_{\mu}\in B_{r_j}$, we have 
	$
	|\lambda-z|
	\geqslant
	|\mu-z|
	-|x_{\mu}|
	\geqslant
	\frac{1}{2}
	(|\mu-z|+|\mu_0-z|)
	-\sqrt{2}
	\geqslant
	\frac{1}{2}
|\mu-\mu_0|-\sqrt{2}
	$. If moreover assume that $\mu\neq \mu_0$,
	then we can continue to estimate
	$|\lambda-z|
	\geqslant
	\frac{1}{2}|\mu-\mu_0|-\sqrt{2}
	\geqslant
	\frac{1}{4}
	|\mu-\mu_0|$, whence
	$|1-\frac{z}{\lambda}|=\frac{|\lambda-z|}{|\lambda|}\geqslant \frac{1}{8}\frac{|\mu-\mu_0|}{|\mu|}$. Since $(\frac{1}{8})^{|B_{r_j}|}\geqslant \exp(-\K\cdot |r_j|^2/c^2)$, 
	we only need to show that 
	\begin{equation}
		\label{core estimate}
	\log
	\prod_{\mu_0\neq\mu\in \A_{r_j}\cap c\Gamma}\frac{|\mu-\mu_0|}{|\mu|}\geqslant 
	-\K \cdot |r_j|^2/c^2
	\geqslant 
	-\K \cdot |z|^2/c^2.
	\end{equation}

	 For any positive number $r'$, denote by $\Gamma_{\leqslant r'}\subset \Gamma$ the subset of points   whose real and imaginary parts have absolute value $\leqslant r'$. Note that, for every $\mu\in \A_{r_j}\cap c\Gamma\setminus(
	 \mu_0+\Gamma_{\leqslant \eta r_j})$ far away from $\mu_0$, we have 
	 \begin{equation}
	 	\label{easy estimate}
	 \frac{|\mu-\mu_0|}{|\mu|}
	>
	\frac{\eta r_j}{r_j}
	= \eta.
	\end{equation}
Thus these $\mu$'s, having cardinality
$|\A_{r_j}\cap c\Gamma\setminus(
\mu_0+\Gamma_{\leqslant \eta r_j})|\leqslant \K\cdot |r_j|^2/c^2$,
 cause no trouble for~\eqref{core estimate}.
	
Lastly, we handle  $\mu\in \A_{r_j}\cap c\Gamma\cap (
\mu_0+\Gamma_{\leqslant \eta r_j})$ simultaneously.
Note that 
$c\Gamma\cap 
\Gamma_{\leqslant \eta r_j}\setminus \{0\}$  can be decomposed 
	into $2$ horizontal parts consisting of $\pm\{\ell\cdot c
	+0\cdot \sqrt{-1}\}_{\ell=1}^{[\frac{\eta}{c}r_j]}$, plus the remaining $4[\frac{\eta}{c} r_j]+2$  vertical parts consisting of
	 $\pm\{i\cdot c
	 +\ell\cdot c\sqrt{-1}\}_{\ell=1}^{[\frac{\eta}{c}r_j]}$
	 for $i=0, \pm 1, \pm 2, \dots, \pm [\frac{\eta}{c} r_j]$.
	 Each part
contains
	consequential $[\frac{\eta}{c} r_j]$ points having absolute values $\geqslant 1\cdot c, 2\cdot c, \dots, [\frac{\eta}{c} r_j]\cdot c$ respectively.
	 Hence 
	 \begin{equation}
	 	\label{square matrix}
	 	\prod_{0\neq \mu'\in c\Gamma\cap 
		\Gamma_{\leqslant \eta r_j}}\,|\mu'|\geqslant
	([\frac{\eta}{c} r_j]!\cdot c^{[\frac{\eta}{c} r_j]})^{4[\frac{\eta}{c} r_j]+2+2}.
\end{equation}
	 Now it is time to apply the Stirling formula that
	for every positive integer $n$, one has
	\[
	n!=n^n e^{-n} \sqrt{2\pi n}\,e^{\rho_n/12n}
	\]
	 for some $|\rho_n|\leqslant 1$.
	 A straightforward computation yields
	 {\footnotesize
	\begin{align}
		\label{Stirling}
	\log
	\prod_{\mu_0\neq\mu
		\in c\Gamma\cap (
		\mu_0+\Gamma_{\leqslant \eta r_j})}
	\frac{|\mu-\mu_0|}{|\mu|}
	&
	\geqslant 
	\log
	\prod_{0\neq \mu'\in c\Gamma\cap 
		\Gamma_{\leqslant \eta r_j}}\,
	\frac{|\mu'|}{2r_j}
	\notag
	\\
	\text{[by~\eqref{square matrix}]}\qquad
	&
	\geqslant 
	\log
	\Big(
	\big(
	[\frac{\eta}{c} r_j]!
	\cdot c^{[\frac{\eta}{c} r_j]}
	\big)^{4[\frac{\eta}{c} r_j]+4}
		\Big)
	-
	\log
	\Big(
	(2r_j)^{4[\frac{\eta}{c} r_j]^2+4[\frac{\eta}{c} r_j]}
	\Big)
	\notag
	\\
	&
	=
	\bigg[
	\log
	\Big(
	\big(
	[\frac{\eta}{c} r_j]!
	\big)^{4[\frac{\eta}{c} r_j]+4}
	\Big)
	-
	\log
	\Big(
	\big(\frac{\eta}{c} r_j
	\big)^{4[\frac{\eta}{c} r_j]^2+4[\frac{\eta}{c} r_j]}
	\Big)
	\bigg]
	+
	\log
	\Big(
	\big(\frac{\eta}{2}
	\big)^{4[\frac{\eta}{c} r_j]^2+4[\frac{\eta}{c} r_j]}
	\Big)
	\
	\notag
	\\
	\text{[by the Stirling formula]}\qquad
	&
	\geqslant
	-\K \cdot |r_j|^2/c^2.
	\end{align}
	
}
	Now the remaining problem is that $c\Gamma\cap (
	\mu_0+\Gamma_{\leqslant \eta r_j})$ might exceed $\mathbf{A}_{r_j}$. 
	Let us decompose $c\Gamma\cap (
	\mu_0+\Gamma_{\leqslant \eta r_j})$ with respect to $\mathbf{A}_{r_j}$ into two parts $(\mu_0+ P_{in})\cup (\mu_0+P_{out})$,
	where the first (resp. second) part lies entirely in (resp. outside) $\mathbf{A}_{r_j}$.
	Since $\eta$ is small,
	we can find some point $y\in \A_{r_j}\cap c\Gamma$ such that $y+P_{out}\subset \mathbf{A}_{r_j}$
	stays  away from $\mu_0+\Gamma_{\leqslant 8\eta r_j}$. See the picture below for illustration.
	
	\begin{center}
		\scalebox{.70}{\begin{picture}(0,0)%
\includegraphics{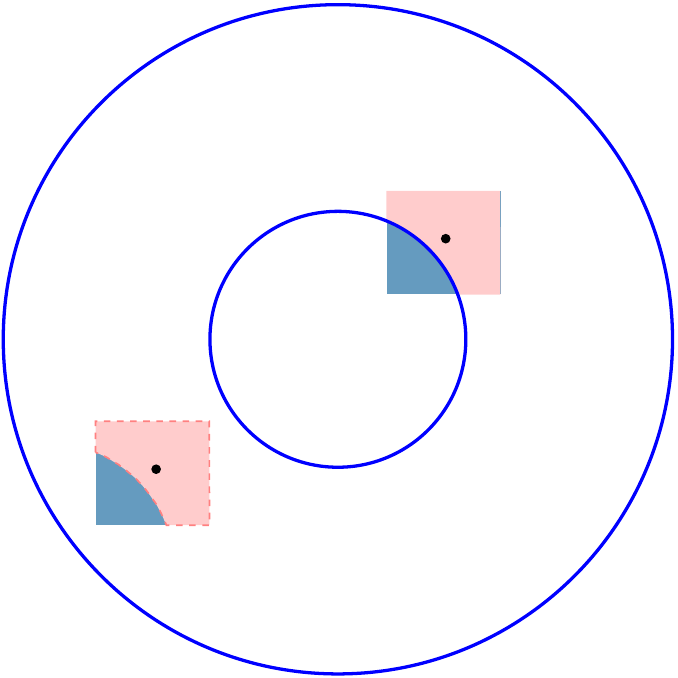}%
\end{picture}%
\setlength{\unitlength}{4144sp}%
\begingroup\makeatletter\ifx\SetFigFont\undefined%
\gdef\SetFigFont#1#2#3#4#5{%
  \reset@font\fontsize{#1}{#2pt}%
  \fontfamily{#3}\fontseries{#4}\fontshape{#5}%
  \selectfont}%
\fi\endgroup%
\begin{picture}(3090,3088)(1336,-2550)
\put(1813,-1977){\makebox(0,0)[lb]{\smash{{\SetFigFont{9}{10.8}{\familydefault}{\mddefault}{\updefault}{\color[rgb]{0,0,0}$y+P_{out}$}%
}}}}
\put(2786,-929){\makebox(0,0)[lb]{\smash{{\SetFigFont{9}{10.8}{\familydefault}{\mddefault}{\updefault}{\color[rgb]{0,0,0}$\mu_0+P_{out}$}%
}}}}
\put(2121,-1602){\makebox(0,0)[lb]{\smash{{\SetFigFont{9}{10.8}{\familydefault}{\mddefault}{\updefault}{\color[rgb]{0,0,0}$y$}%
}}}}
\put(3365,-247){\makebox(0,0)[lb]{\smash{{\SetFigFont{9}{10.8}{\familydefault}{\mddefault}{\updefault}{\color[rgb]{0,0,0}$\mu_0+P_{in}$}%
}}}}
\put(4124,-2076){\makebox(0,0)[lb]{\smash{{\SetFigFont{9}{10.8}{\familydefault}{\mddefault}{\updefault}{\color[rgb]{0,0,0}$r_j$}%
}}}}
\put(3276,-1621){\makebox(0,0)[lb]{\smash{{\SetFigFont{9}{10.8}{\familydefault}{\mddefault}{\updefault}{\color[rgb]{0,0,0}$\frac{r_j}{2}$}%
}}}}
\put(3426,-556){\makebox(0,0)[lb]{\smash{{\SetFigFont{9}{10.8}{\familydefault}{\mddefault}{\updefault}{\color[rgb]{0,0,0}$\mu_0$}%
}}}}
\end{picture}%
}
	\end{center}

 Lastly, we decompose 
	$\A_{r_j}\cap c\Gamma$ into $3$ disjoint parts,
	$\mu_0+ P_{in}$,
	$y+ P_{out}$ and $R$ the remaining.	Note that $\prod_{\mu\in y+P_{out}}\frac{|\mu-\mu_0|}{|\mu|}
	\geqslant
	\prod_{\mu\in \mu_0+P_{out}}\frac{|\mu-\mu_0|}{|\mu|}
	$, because each factor on the left-hand-side
	$\geqslant \frac{8\eta r_j}{r_j}=8\eta$, while each factor on the right-hand-side
	$\leqslant \frac{\sqrt{2}\eta r_j}{r_i/4}=4\sqrt{2}\eta$. Thus
	$
	\prod_{\mu_0\neq\mu\in \A_{r_j}\cap c\Gamma}\frac{|\mu-\mu_0|}{|\mu|}
	\geqslant
	\prod_{\mu_0\neq\mu
		\in c\Gamma\cap (
		\mu_0+\Gamma_{\leqslant \eta r_j})}
	\frac{|\mu-\mu_0|}{|\mu|}
	\cdot
	\prod_{\mu\in R}
	\frac{|\mu-\mu_0|}{|\mu|}
	$.
	By the estimates~\eqref{easy estimate},~\eqref{Stirling} and that the cardinality $| R|\leqslant  \K\cdot  r_j^2/c^2$,
	we conclude the proof.
\end{proof}
	
\section{\bf Estimates}
\label{section: estimate area of discs}	

\subsection{$f(z)$ is close to $\mathcal{C}_{\infty}$ unless $z$ is near $\Lambda$}
\label{f(z) near infinity curve}
Recalling~\eqref{psi_1},
we first rewrite
\[
||\psi\cdot s_m(z)||_{h_m}
=
||\psi_1\cdot s_m(z)||_{h_m}
\cdot 
|\diamondsuit|
\]
to concentrate positivity to the first factor,
where
if $z\in \D(\lambda, 1)$ for certain $\lambda\in \Lambda$ then $\diamondsuit=(1-\frac{z}{\lambda})e^{\frac{z}{\lambda}+\frac{z^2}{2\lambda^2}}$, otherwise $\diamondsuit=1$.
Now thanks to~\eqref{exponential growth of s},~\eqref{psi>},
the left part satisfies that
\[
||\psi_1\cdot s_m(z)||_{h_m}
\geqslant
\exp
\big(
(m\cdot\alpha-\K/c^2)
\cdot
|z|^2
\big).
\]
A key trick in this paper is that we choose $m$ and $c$ such that
\begin{equation}
\label{key choices of m, k}
m\cdot\alpha-\K/c^2
>0.
\end{equation}
Thus
for any small $\epsilon>0$,
for all sufficiently large $|z|\gg 1$ with $\text{dist}(z, \Lambda)\geqslant \epsilon$,
there holds
\begin{equation}
\label{psi.s large}
||\psi\cdot s_m(z)||_{h_m}
\gg
1,
\end{equation} 
i.e., $f(z)$ is very close to $\mathcal{C}_{\infty}$.
Indeed, if $\text{dist}(d, \Lambda)\geqslant 1$, then $\diamondsuit=1$, and there is nothing to proof; otherwise $\epsilon\leqslant |z-\lambda|< 1$ for some $\lambda\in \Lambda$,
hence
$
|\diamondsuit|
\geqslant
\frac{|z-\lambda|}{|\lambda|}
\cdot
\exp(-\frac{|z|}{|\lambda|}-\frac{|z|^2}{2|\lambda|^2})
\geqslant
\frac{\epsilon}{|z|+1}
\cdot
\exp(-\K)
$, thus
$||\psi\cdot s(z)||_{h_m}=||\psi_1\cdot s(z)||_{h_m}
\cdot 
|\diamondsuit|$ is very large when  $|z|\gg 1$.

\subsection{Bound the area $\int_{\D_{2r_i}}{f}^*\omega_{X}$ from above}

Fix a  K\"ahler form $\omega_{\mathcal{C}}=\dif\dif^c\big(|z|^2\big)$ on $\mathcal{C}$.  
The metric $h_m=h^{\otimes m}$ of $\mathcal{L}_m=\mathcal{L}^{\otimes m}$ together with the standard Euclidean metric
$|dz|$ on $\mathbb{C}$ provide a metric for the vector bundle $\mathcal{E}:=\mathcal{L}_m\oplus \mathbb{C}$, and therefore it induces a metric  on the tautological line bundle $\mathcal{O}_{\mathbb{P}(\mathcal{E})}(-1)$ on $\mathbb{P}(\mathcal{E})=X$.
Restricting to any fiber of 
$\pi_1: X\longrightarrow \mathcal{C}$,
the curvature form $\Theta_{\mathcal{O}_{\mathbb{P}(\mathcal{E})}(-1)}$ is strictly negative due to the property of the Fubini-Study metric for the tautological line bundle $\mathcal{O}_{\mathbb{P}^1}(-1)$. Therefore, by standard compactness argument, for sufficiently small $\epsilon_1>0$, we receive a 
K\"ahler form on $X$ of the shape
\begin{equation}
	\label{kahler form on X}
	\omega_X
	:=
	\pi_1^*\omega_{\mathcal{C}}-\epsilon_1\, \Theta_{{O}_{\mathbb{P}(\mathcal{E})}(-1)}.
\end{equation}

We can identify the tautological space $\mathcal{L}_m=\{(z, \xi): \xi\in \mathcal{L}_m|_{z}\}$ with an open set of $ \mathbb{P}(\mathcal{L}_m\oplus \mathbb{C})$, by
mapping $(z, \xi)\mapsto (z, [\xi\oplus 1])$. Thus in the local coordinates $(z, \xi)$, 
 the curvature
\[
\Theta_{\mathcal{O}_{\mathbb{P}(\mathcal{E})}(-1)}=
-\dif\dif^c\big(
\log(\|\xi\|_{h_m}^2+1)
\big)
\]
 is of Fubini-Study shape.
For $r$ lies in $[\frac{1}{3}r_i, 3r_i]$ for some $i\geqslant 1$, by Jensen's formula
we receive
{\footnotesize
\begin{align}
	\label{Jenson formula for Funiby-Study}
\int_{1}^r\frac{\dif t}{t}\int_{\D_t}
f^*
\Theta_{\mathcal{O}_{\mathbb{P}(\mathcal{E})}(-1)}
&
=
-
\frac{1}{4\pi}\,
\int_{0}^{2\pi}
\log(\|\psi\cdot s_m\|^2_{h_m}+1)(re^{i\theta})\dif \theta
+
\frac{1}{4\pi}\,
\int_{0}^{2\pi}
\log(\|\psi\cdot s_m\|^2_{h_m}+1)(e^{i\theta})\dif \theta
\notag
\\
\text{[use~\eqref{exponential growth of s},~\eqref{psi<}]}\qquad
&
\geqslant
-\K\cdot r_i^2
\notag,
\end{align}
}
 whence the Nevanlinna order function
 satisfies the estimate
\begin{equation}
	\label{order function estimate <=}
T_{f, {r}}(\omega_X)
=\int_{1}^{r}\frac{\dif t}{t}\int_{\D_t}f^*\omega_X
=
\int_{1}^{r}\frac{\dif t}{t}\int_{\D_t}f^*\pi_1^*\,\omega_{\mathcal{C}}
-
\epsilon_1\,
\int_{1}^{r}\frac{\dif t}{t}\int_{\D_t}f^*\Theta_{\mathcal{O}_{\mathbb{P}(\mathcal{E})}(-1)}
\leqslant 
\K\cdot r_i^2.
\end{equation}
Here is a useful observation
\begin{equation}
	\label{easy trick}
T_{{f},{3r_i}}(\omega_{X})
=\int_{1}^{{3r_i}}\frac{\dif t}{t}\int_{\D_t}{f}^*\omega_{X}
\geqslant \int_{{2r_i}}^{{3r_i}}\frac{\dif t}{t}\int_{\D_{2r_i}}{f}^*\omega_{X}
=
\ln (3/2) \cdot
\int_{\D_{2r_i}}
{f}^*\omega_{X}.
\end{equation}
Combing the two estimates above, we conclude that
\begin{equation}
	\label{disc image area is bounded}
\int_{\D_{2r_i}}{f}^*\omega_{X}\leqslant\K\cdot r_i^2.
\end{equation}

\subsection{Bound the area $\int_{\D_{r_i/3}}{f}^*\omega_{X}$ from below}
By Jensen's formula,
we have
{\footnotesize
	\begin{align}
		\label{tricky Jensen}
		\int_{r_i/4}^{r_i/3}\frac{\dif t}{t}\int_{\D_t}
		f^*
		\Theta_{\mathcal{O}_{\mathbb{P}(\mathcal{E})}(-1)}
		&
		=
		-
		\frac{1}{4\pi}\,
		\int_{0}^{2\pi}
		\log(\|\psi\cdot s_m\|^2_{h_m}+1)(\frac{r_i}{3}e^{i\theta})\dif \theta
		+
		\frac{1}{4\pi}\,
		\int_{0}^{2\pi}
		\log(\|\psi\cdot s_m\|^2_{h_m}+1)(\frac{r_i}{4}e^{i\theta})\dif \theta
		\notag
		\\
		\text{[for $r_i\gg 1$]}
		\qquad
		&
		=
		-
		\frac{1}{4\pi}\,
		\int_{0}^{2\pi}
		\log(\|\psi\cdot s_m\|^2_{h_m})(\frac{r_i}{3}e^{i\theta})\dif \theta
		+
		\frac{1}{4\pi}\,
		\int_{0}^{2\pi}
		\log(\|\psi\cdot s_m\|^2_{h_m})(\frac{r_i}{4}e^{i\theta})\dif \theta
		+
		o(1).
	\end{align}
}

Since the holomorphic function $\psi$ is nowhere vanishing on $\overline{\D}_{r_i/3}\setminus \D_{r_i/4}$ by our construction,
 $\log |\psi|^2$
is harmonic on $\overline{\D}_{r_i/3}\setminus \D_{r_i/4}$. Thus
\[
-
\frac{1}{4\pi}\,
\int_{0}^{2\pi}
\log |\psi|^2(\frac{r_i}{3}e^{i\theta})\dif \theta
+
\frac{1}{4\pi}\,
\int_{0}^{2\pi}
\log |\psi|^2(\frac{r_i}{4}e^{i\theta})\dif \theta
=
0.
\]
Hence we can simplify~\eqref{tricky Jensen} as
{\footnotesize
	\begin{align*}
		\label{tricky Jensen 2}
		\int_{r_i/4}^{r_i/3}\frac{\dif t}{t}\int_{\D_t}
		f^*
		\Theta_{\mathcal{O}_{\mathbb{P}(\mathcal{E})}(-1)}
		&
		=
		-
		\frac{1}{4\pi}\,
		\int_{0}^{2\pi}
		\log(\| s_m\|^2_{h_m})(\frac{r_i}{3}e^{i\theta})\dif \theta
		+
		\frac{1}{4\pi}\,
		\int_{0}^{2\pi}
		\log(\| s_m\|^2_{h_m})(\frac{r_i}{4}e^{i\theta})\dif \theta
		+
		o(1)
		\notag
		\\
		\text{[by~\eqref{exponential growth of s}]}\qquad
		&
		=
		-\frac{1}{4\pi}\,
		\int_{0}^{2\pi}
		2m\alpha\bigg|\frac{r_i}{3}\bigg|^2
		+\frac{1}{4\pi}\,
		\int_{0}^{2\pi}
		2m\alpha\bigg|\frac{r_i}{4}\bigg|^2
		+o(1)
		\notag\\
		&
		=
		-
		\frac{7m\alpha}{144}r_i^2
		+
		o(1).
	\end{align*}
}
Hence
{\footnotesize
\begin{equation}
	\label{key estimates about big dics}
\int_{r_i/4}^{r_i/3}\frac{\dif t}{t}\int_{\D_t}f^*\omega_X
=
\int_{r_i/4}^{r_i/3}\frac{\dif t}{t}\int_{\D_t}f^*\pi_1^*\omega_{\mathcal{C}}
-
\epsilon_1
\int_{r_i/4}^{r_i/3}\frac{\dif t}{t}\int_{\D_t}
f^*
\Theta_{\mathcal{O}_{\mathbb{P}(\mathcal{E})}(-1)}
=
\frac{7}{144}
(\frac{\pi}{2}+
\epsilon_1 m\alpha)r_i^2
+
o(1).
\end{equation}
}
Noting that
\[
\int_{r_i/4}^{r_i/3}\frac{\dif t}{t}\int_{\D_t}f^*\omega_X
\leqslant
\int_{r_i/4}^{r_i/3}\frac{\dif t}{t}\int_{\D_{r_i/3}}f^*\omega_X
=
\log(4/3)\cdot
\int_{\D_{r_i/3}}f^*\omega_X,
\]
we conclude that for $i\gg 1$ there holds
\begin{equation}
	\label{disc area is large enough}
	\int_{\D_{r_i/3}}f^*\omega_X
	\geqslant
	\K\cdot r_i^2.
\end{equation}

\subsection{Estimates of $
	\int_{\mathbb{D}(\lambda,\epsilon)}f^*\omega_X
	$}  

Mark the curve $\mathcal{C}_0\subset X$ induced by the zero section of $\mathcal{L}_m$. Note that $\mathcal{C}_0$ and $\mathcal{C}_{\infty}$ are disjoint. Contrasting to the observation in Subsection~\ref{f(z) near infinity curve},
for every $\lambda\in \Lambda$, since $\psi(\lambda)=0$, $f(\lambda)$ must lies in $\mathcal{C}_0$, which keeps certain positive distance to $\mathcal{C}_{\infty}$.

 Let $\epsilon$
be a small positive radius.
Denote by 
$\D([\lambda], \epsilon)\subset \mathcal{C}$ the disc centered at $[\lambda]$ with the radius $\epsilon$. Then the image of ${f}\big(\D({\lambda, \epsilon})\big)$ is contained in the neighborhood 
$U_{[\lambda], \epsilon}:=
\pi_1^{-1}
(\D([\lambda], \epsilon))
$
of $\PP_{[\lambda]}:=\pi_1^{-1}([\lambda])$.
Recall that the metric $\omega_X$ on $X$ is given in~\eqref{kahler form on X}. Firstly, since $\pi_1 \circ f=\pi_0$, where  $\pi_0:\mathbb{C}\rightarrow\mathbb{C}/\Gamma$ is the canonical projection, we have
$
\int_{\mathbb{D}(\lambda,\epsilon)}f^*\pi_1^*\omega_{\mathcal{C}}
=
\pi\epsilon^2$.

Next, we estimate $
\int_{\mathbb{D}(\lambda,\epsilon)}f^*\Theta_{{O}_{\mathbb{P}(\mathcal{E})}(-1)}
$.
Recalling~\eqref{psi_1}, 
we can rewrite
$
\psi(z)=(
		\psi_1\cdot e^{\frac{z}{\lambda}+\frac{z^2}{2\lambda}}
		)
\cdot
(
		1
		-
		\frac{z}{\lambda}
		)
		$,
		where the  factor $\psi_2:=	\psi_1\cdot e^{\frac{z}{\lambda}+\frac{z^2}{2\lambda}}$
		is nonvanishing for $z\in \D(\lambda, \epsilon)$. Thus $\log|\psi_2|$ is a harmonic function on $\D(\lambda, \epsilon)$. Now, using Jensen's formula, for $|\lambda|\gg 1$ large, we can estimate 
		{\footnotesize
\begin{align}
\label{area of disc D(lambda,t), O P()}
\int_{\epsilon}^{2\epsilon}
\frac{\dif t}{t}\int_{\mathbb{D}(\lambda,t)}
f^*
\Theta_{\mathcal{O}_{\mathbb{P}(\mathcal{E})}(-1)}
&=
-
\frac{1}{4\pi}\,
\int_{0}^{2\pi}
\log(\|\psi\cdot s_m\|^2_{h_m}+1)(\lambda+2\epsilon e^{i\theta})\dif \theta
+
\frac{1}{4\pi}\,
\int_{0}^{2\pi}
\log(\|\psi^2\cdot s_m\|^2_{h_m}+1)(\lambda+\epsilon e^{i\theta})\dif \theta\notag\\
\text{[by~\eqref{psi.s large}, as $|\lambda|\gg 1$]}
\qquad
&
=
-
\frac{1}{4\pi}\,
\int_{0}^{2\pi}
\log(\|\psi\cdot s_m\|^2_{h_m})(\lambda+2\epsilon e^{i\theta})\dif \theta
+
\frac{1}{4\pi}\,
\int_{0}^{2\pi}
\log(\|\psi^2\cdot s_m\|^2_{h_m})(\lambda+\epsilon e^{i\theta})\dif \theta +o(1)
\notag\\
&=
\bigg[-
\frac{1}{4\pi}\,
\int_{0}^{2\pi}
\log(\| s_m\|^2_{h_m})(\lambda+2\epsilon e^{i\theta})\dif \theta+
\frac{1}{4\pi}\,
\int_{0}^{2\pi}
\log(\| s_m\|^2_{h_m})(\lambda+\epsilon e^{i\theta})\dif \theta\bigg]
\\
&
\qquad
+
\bigg[-
\frac{1}{4\pi}\,
\int_{0}^{2\pi}
\log| \psi|^2(\lambda+2\epsilon e^{i\theta})\dif \theta+
\frac{1}{4\pi}\,
\int_{0}^{2\pi}
\log| \psi|^2(\lambda+\epsilon e^{i\theta})\dif \theta\bigg]
+o(1).
\notag
\end{align}
}

By \eqref{exponential growth of s}, the first bracket $[\cdots]$ in \eqref{area of disc D(lambda,t), O P()} can be computed as 
\begin{align}
\label{first term, norm s}
-\frac{1}{4\pi}\,
\int_{0}^{2\pi}
2m
\alpha|\lambda+2\epsilon e^{i\theta}|^2\dif \theta
+
\frac{1}{4\pi}\,
\int_{0}^{2\pi}
2m
\alpha|\lambda+\epsilon e^{i\theta}|^2\dif \theta
=
\dfrac{-3 m
	\alpha\epsilon^2}{2\pi}.
\end{align}

Now we separate $\log|\psi|(z)=\log|\psi_2|\cdot \log|1-\frac{z}{\lambda}|$. Thanks to the harmonicity of
$\log|\psi_2|$, we have 
\[
-
\frac{1}{4\pi}\,
\int_{0}^{2\pi}
\log(| \psi_2|^2)(\lambda+2\epsilon e^{i\theta})\dif \theta+
\frac{1}{4\pi}\,
\int_{0}^{2\pi}
\log(| \psi_2|^2)(\lambda+\epsilon e^{i\theta})\dif \theta
=
0.
\]
Thus we can calculate the second  bracket $[\cdots]$ of \eqref{area of disc D(lambda,t), O P()} as
\begin{align}
\label{second term, norm psi}
-
\frac{1}{4\pi}\,
\int_{0}^{2\pi}
\log
\bigg|
1
-
\frac{\lambda+2\epsilon e^{i\theta}}{\lambda}
\bigg|^2
\dif \theta
+
\frac{1}{4\pi}\,
\int_{0}^{2\pi}
\log
\bigg|
1
-
\frac{\lambda+\epsilon e^{i\theta}}{\lambda}
\bigg|^2
\dif \theta
=
-\log 2.
\end{align}

Hence it follows from \eqref{area of disc D(lambda,t), O P()}, \eqref{first term, norm s}, \eqref{second term, norm psi} that
\[
\int_{\epsilon}^{2\epsilon}
\frac{\dif t}{t}\int_{\mathbb{D}(\lambda,t)}
f^*
\Theta_{\mathcal{O}_{\mathbb{P}(\mathcal{E})}(-1)}
=
-\dfrac{3m\alpha\epsilon^2}{2\pi}
-\log 2
+
o(1)
\qquad
{\scriptstyle(\text{for }|\lambda|\gg 1)}.
\]
Therefore
\begin{equation}
\label{estimate are of disc lambda epsilon, epsilon--2epsilon}
\int_{\epsilon}^{2\epsilon}
\frac{\dif t}{t}\int_{\mathbb{D}(\lambda,t)}
f^*
\omega_X
=
\pi\epsilon^2
+\epsilon_1
\bigg(
\dfrac{3m\alpha\epsilon^2}{2\pi}
+
\log 2
+
o(1)
\bigg).
\end{equation}
By the same trick as~\eqref{easy trick}, we have
\[
\int_{\epsilon}^{2\epsilon}
\frac{\dif t}{t}\int_{\mathbb{D}(\lambda,t)}
f^*
\omega_X
\geqslant
\int_{\epsilon}^{2\epsilon}
\frac{\dif t}{t}\int_{\mathbb{D}(\lambda,\epsilon)}
f^*
\omega_X
\geqslant
\log 2\cdot \int_{\mathbb{D}(\lambda,\epsilon)}
f^*
\omega_X.
\]
Combining the above two estimates, we conclude
\begin{equation}
	\label{small disc high bound}
\int_{\mathbb{D}(\lambda,\epsilon)}
f^*
\omega_X
\leqslant
\dfrac{1}{\log 2}
\bigg(
\pi\epsilon^2
+\epsilon_1
\Big(
\dfrac{3m\alpha\epsilon^2}{2\pi}
+
\log 2
+
o(1)
\Big)
\bigg)
\qquad
{\scriptstyle(\text{for }|\lambda|\gg 1)}.
\end{equation}

Next, we provide an lower bound for $\int_{\mathbb{D}(\lambda,\epsilon)}
f^*\omega_X$. Substituting $\epsilon$ by $\frac{\epsilon}{2}$ in \eqref{estimate are of disc lambda epsilon, epsilon--2epsilon}, we receive that
\begin{equation}
\label{estimate are of disc lambda epsilon, 1/2epsilon--epsilon}
\int_{\frac{\epsilon}{2}}^{\epsilon}
\frac{\dif t}{t}\int_{\mathbb{D}(\lambda,t)}
f^*
\Theta_{\mathcal{O}_{\mathbb{P}(\mathcal{E})}(-1)}
=
\frac{\pi\epsilon^2}{4}
+\epsilon_1
\bigg(
\dfrac{3m\alpha\epsilon^2}{8\pi}
+
\log 2
+
o(1)
\bigg).
\end{equation}
Note that
\[
\int_{\frac{\epsilon}{2}}^{\epsilon}
\frac{\dif t}{t}\int_{\mathbb{D}(\lambda,t)}
f^*
\omega_X
\leqslant
\int_{\frac{\epsilon}{2}}^{\epsilon}
\frac{\dif t}{t}\int_{\mathbb{D}(\lambda,\epsilon)}
f^*
\omega_X
=
\log 2\cdot \int_{\mathbb{D}(\lambda,\epsilon)}
f^*
\omega_X.
\]
Hence it follows from the above two estimates that
\begin{equation}
	\label{small disc low bound}
\int_{\mathbb{D}(\lambda,\epsilon)}
f^*
\omega_X
\geqslant
\dfrac{1}{\log 2}
\bigg(
\frac{\pi\epsilon^2}{4}
+\epsilon_1
\Big(
\dfrac{3m\alpha\epsilon^2}{8\pi}
+
\log 2
+
o(1)
\Big)
\bigg)
\qquad
{\scriptstyle(\text{for }|\lambda|\gg 1)}.
\end{equation}

\subsection{Area estimates of $f(\CC)$ near horizontal curves}
\label{estimate area near horizontal curves}
An irreducible algebraic curve $D\subset X$
is said to be vertical if
$\pi_1(D)$ is a point; otherwise it is called horizontal, in the sense that $\pi_1(D)=\mathcal{C}$.

Firstly, for a vertical curve $\mathbb{P}^1_{[y]}=\pi_1^{-1}([y])$, by the estimates~\eqref{small disc high bound} and~\eqref{small disc low bound}, the area of $f(\D_{r})$  near $\mathbb{P}^1_{[y]}$, 
as $r\rightarrow \infty$,
is mostly
determined by 
asymptotic
growth of $|\D_r\cap \Lambda|$.

Next, for the horizontal curve $D=\mathcal{C}_{\infty}$,
by Subsection~\ref{f(z) near infinity curve}, 
$f(\D_{r})$ shall concentrate a large portion of area near $\mathcal{C}_{\infty}$
as $r\rightarrow \infty$.

Lastly, for any other irreducible horizontal curve $D\not= \mathcal{C}_{\infty}$,
we devote this subsection 
to prove that,
roughly speaking,
every time when
$f(\D_r)$ intersects with $D$, it contributes negligible area about there.

\smallskip

To start with, we take a general point $d_0\in D\setminus \mathcal{C}_{\infty}$
such that $\pi_1|_{D}$
is regular at $d_0$, i.e., some small open neighborhood  $U\subset D $ of $d_0$ is a graph over $\pi_1(U)$ containing $ \pi_1(d_0)=:c_0$.
By shrinking $U$ we may assume that $U$ stays away from $\mathcal{C}_{\infty}$, that  $\pi_1(U)$ is a small disc $\D(c_0, 3\delta)$ for some $\delta>0$, and that
the line bundle $\mathcal{L}_m$ locally has a trivialization
$\mathcal{L}_m|_{\D(c_0, 3\delta)}
\cong
\D(c_0, 3\delta)\times \CC$, which extends to an identification 
\begin{equation}
	\label{local chart of X}
	\vartheta:
	\pi_1^{-1}(\D(c_0, 3\delta))
	\xrightarrow[]{\cong}
	\D(c_0, 3\delta)\times \mathbb{P}^1(\mathbb{C})
\end{equation}
by fiberwised 
compactification
$\mathbb{C}\hookrightarrow \mathbb{P}^1(\mathbb{C})$  sending $z \mapsto [z: 1]$.
Hence we can read the coordinates of $U$ in the chart  $ \D(c_0, 3\delta)\times \CC$
as the graph of
some holomorphic map
$u: \D(c_0, 3\delta)\rightarrow \CC$.

Let $p_1, p_2$ be the  projections of 
$\D(c_0, 3\delta)\times \mathbb{P}^1(\mathbb{C})$ to the two factors.
Let $\omega_{\FS}$ be the Fubini-Study form on $\mathbb{P}^1(\mathbb{C})$.
By compactness argument,
the metric $p_1^*\omega_{\mathcal{C}}+p_2^*\omega_{\FS}$ is comparable with $(\vartheta^{-1})^*\omega_X$ on $\D(c_0, \frac{5\delta}{2})\times \mathbb{P}^1(\mathbb{C})$, namely
\begin{equation}
	\label{comparable metrics}
\K_{c_0, \delta, \vartheta}^{-1}\cdot(p_1^*\omega_{\mathcal{C}}+p_2^*\omega_{\FS})\leqslant
(\vartheta^{-1})^*\omega_X
\leqslant \K_{c_0, \delta, \vartheta}\cdot(p_1^*\omega_{\mathcal{C}}+p_2^*\omega_{\FS}).
\end{equation}

Fix a positive number
$\epsilon\ll \delta$.
Then the neighborhood $\pi_1^{-1}(\D(c_0, \delta))\cap D$ of $d_0$ in the coordinates reads as
\[
U_1
:=
\{
(z, w) :
z\in \D(c_0, \delta),
w=u(z)
\},
\] 
whose $\epsilon$--open neighborhood is
\[
U_{1}^{\epsilon}
:=
\{(z, w): z\in\D(c_0, \delta+\epsilon), |w-u(z)|<\epsilon\}.
\]

Fix a positive small number
$\delta'<\delta/2$.
By  Subsection~\ref{f(z) near infinity curve},
for $|z|\gg 1$ large with
dist$(z, \Lambda)>\delta'$, one sees that
$f(z)$ is very close to $\mathcal{C}_{\infty}$,
hence it is outside $U_1^{\epsilon}$.
Thus for bounding the area of 
$f(\D_r)\cap U_1^{\epsilon}$ from above by $o(1)\cdot r^2$,
we only need to show that,
as $\lambda\in \Lambda$ with $|\lambda|\gg 1$, the area 
$f(\D(\lambda, \delta'))\cap U_1^{\epsilon}$ is negligible $o(1)$.

\begin{obs}
	\label{shrink radius}
	Put $f_2:=p_2\circ \vartheta\circ f$.
	For $\lambda\in\Lambda$  with $|\lambda|\gg 1$ and 
	$[\lambda]\in \D(c_0, \delta+\delta')$,
	one has
	\[
	(\vartheta\circ f)^{-1}(U_1^{\epsilon})
	\cap
	\D(\lambda, \delta')
	\subset
	{f_2}^{-1}
	\Big(\mathbb{D}
	\big(
	u([\lambda]), 2\epsilon\big)
	\Big)
	\cap
	\mathbb{D}(\lambda, \delta').
	\]
	
\end{obs}

\begin{proof}
	By continuity of $u$ and by compactness of $\overline\D(c_0, \frac{5}{2}\delta)$, there exists 
	some positive number $\delta_{\epsilon}<\delta'$
	such that, for any two points $x_1, x_2\in 
	\overline\D(c_0, \frac{5}{2}\delta)$ with $|x_1-x_2|< \delta_{\epsilon}$,
	there holds
	$|u(x_1)-u(x_2)|<\epsilon$. By  Subsection~\ref{f(z) near infinity curve},
	for all $\lambda
	\in \Lambda$ with $|\lambda|\gg1$, the image of 
	$\D(\lambda, \delta')\setminus 
	{\D}(\lambda, \delta_{\epsilon})$
	under $\vartheta\circ f$ is outside $U_1^{\epsilon}$, therefore
	\[
	(\vartheta\circ f)^{-1}(U_1^{\epsilon})
	\cap
	\D(\lambda, \delta')
	\subset 
	(\vartheta\circ f)^{-1}(U_1^{\epsilon})
	\cap
	\D(\lambda, \delta_{\epsilon}).
	\]
	By definition, every element
	$z$ in the right-hand-side satisfies that
	$|f_2(z)-u([z])|<\epsilon$ and
	$|z-\lambda|< \delta_{\epsilon}$.
	Thus $|f_2(z)-u([\lambda])|
	\leqslant
	|f_2(z)-u([z])|+ |u([z])-u([\lambda])|
	< \epsilon+\epsilon=2\epsilon$,
	which concludes the proof.
\end{proof}

Now for every $v\in \mathbb{C}$ having absolute value $|c|\leqslant R:=\max\{|u(z)|+\epsilon:
z\in\overline\D(c_0, \frac{5}{2}\delta)\}<\infty$, for  $\lambda\in \Lambda$ with $|\lambda|\gg1$ and
$[\lambda]\in \D(c_0, \delta+\delta')$, consider the restricted holomorphic function
\[
f_2: \mathbb{D}(\lambda, \delta')\longrightarrow \mathbb{C}.
\]
Since $|f_2|\gg 1$  (in particular $|f_2|>R$) on $\partial \mathbb{D}(\lambda, \delta')$ by  Subsection~\ref{f(z) near infinity curve}, by the Argument Principle, the number of solutions of the equation
$f_2(y)=v$ on the disc $\mathbb{D}(\lambda, \delta')$, counting multiplicities, equals to
\[
\frac{1}{2\pi \sqrt{-1}}
\int_{z\in \partial \mathbb{D}(\lambda, \delta')}
\frac{(f_2-v)'}{f_2-v}(z)
\dif z.
\]
Noting that the above quantity takes integer value,
and that it varies continuously with respect to $v\in \D_R$,
it must be a constant for every $v\in \D_R$. Now checking the special value $v=0$, we see that
the number of solution is just $1$. Thus Observation~\ref{shrink radius} implies that,  for
$\lambda\in\Lambda$
with $|\lambda|$ sufficiently large and $[\lambda]\in \D(c_0, \delta+\delta')$,
we have
\begin{equation}
\label{inequality 1}
\area\Big(
(\vartheta\circ f)^{-1}(U_1^{\epsilon})
\cap
\D(\lambda, \delta')
\Big)_{f_2^*\omega_{\FS}}
\leqslant
\area
\big(
u([\lambda]), 2\epsilon\big)_{\omega_{\FS}}
\leqslant
\K\cdot \epsilon^2.
\end{equation}
Also, by Subsection~\ref{f(z) near infinity curve},
for every positive number $\epsilon'>0$,
for  $\lambda\in \Lambda$ with sufficiently large $|\lambda|$,
we have
\[
(\vartheta\circ f)^{-1}(U_1^{\epsilon})
\cap
\D(\lambda, \delta')
\subset
\D(\lambda, \epsilon').
\]
Therefore
\begin{equation}
\label{inequality 2}
\area\Big(
(\vartheta\circ f)^{-1}(U_1^{\epsilon})
\cap
\D(\lambda, \delta')
\Big)_{\pi_1^*\omega_{\mathcal{C}}}
\leqslant
\area
\Big(
\D(\lambda, \epsilon')
\Big)_{\pi_1^*\omega_{\mathcal{C}}}
=
\pi\cdot\epsilon'^2.
\end{equation}
Summarizing~\eqref{comparable metrics},~\eqref{inequality 1},~\eqref{inequality 2}, for 
 $\lambda \in \Lambda$ with $|\lambda|\gg1$ and $[\lambda]\in \D(c_0, \delta+\delta')$, we have
\begin{equation}
\label{very subtle estimate}
\area\Big(
(\vartheta\circ f)^{-1}(U_1^{\epsilon})
\cap
\D(\lambda, \delta')
\Big)_{f^*\omega_{X}}
\leqslant
\K_{c_0, \delta,  \vartheta}\cdot
(\K\,
\epsilon^2
+
\pi\,\epsilon'^2
).
\end{equation}

\subsection{Area of $f$ near $\lambda\in \Lambda$ revisit}

An alternative way to interpret~\eqref{small disc low bound} is the following
 
 \begin{obs}
 	\label{obs 4.2}
 	Let $\delta'>0$ be a small positive number. Let
 	 $U$ be an open neighborhood of $\mathcal{C}_{\infty}$ such that its closure $\overline{U}$ stays away from $\mathcal{C}_0$. 
 	 Then one has the estimate
	\begin{equation}
		\label{very subtle estimate again}
		\area\Big(
		(X\setminus U)
		\cap
		f\big(
		\D(\lambda, \delta')
		\big)
		\Big)_{\omega_{X}}
		\geqslant
		\K_U
		\qquad
		{
			 (\forall \lambda\in \Lambda\text{ with } |\lambda|
			 \gg 1)}.
	\end{equation}
 \end{obs}

This strengthens~\eqref{small disc low bound}
in a qualitative sense,
and will be helpful for discussing diffuse parts later. Before going to the proof of the above result, recall the following special case of Wirtinger's inequality.
	
	\begin{pro}
		\label{curve in ball area}
		\text{(c.f.~ \cite[page~7]{Duval2017-2})} 
	Let $C$ be a proper holomorphic curve in the ball
	$B(0, \epsilon)\subset \mathbb{C}^n$
passing through $0$.
Then with the standard Euclidean metric, one has
$
\area(C)\geqslant \pi \epsilon^2.$
\qed
\end{pro}

\begin{proof}[Proof of Observation~\ref{obs 4.2}]
By compactness,  $\mathcal{C}_0$ can be covered by finitely many open neighborhoods $U_i$, being disjoint with $U$, with charts
$\vartheta_i: U_i\rightarrow V_i\subset \mathbb{C}^2$.
By shrinking $U_i$'s if necessary, we can assume that every pull-back by $\vartheta_i$ of the standard Euclidean metric on $\mathbb{C}^2$ is comparable with $\omega_X$. Again by the compactness of $\mathcal{C}_0$,
for every point $c\in \mathcal{C}_0$, certain chart $V_i$ of
$U_i\owns c$ contains a sufficiently large ball centered at
$\vartheta_i(c)$
with a uniform radius $r>0$.
Now, by Subsection~\ref{f(z) near infinity curve}, for $\lambda\in \Lambda$ with $|\lambda|\gg1$, 
for $c=f(\lambda)\in \mathcal{C}_0$, in the chart $V_i$ we see that 
$\vartheta_i\big(	f(
\D(\lambda, \delta')
)\cap U_i\big)$ 
contains a proper holomorphic curve in the ball
$B(\vartheta_i(c), r)$,
having positive area $\geqslant \pi r^2$
 by Proposition~\ref{curve in ball area}. The desired conclusion then  follows from the comparability of metrics.
\end{proof}

\section{\bf Algorithm}
\label{section: algorithm}

First of all, we require that $m, c$ satisfy the condition~\eqref{key choices of m, k}.

Next, we choose  distinct points in
a strip of $\mathcal{D}$
\begin{equation}
	\label{choose y_i}
\{y_i\}_{i\in \mathbb{Z}_+}
\subset
\{x+y\sqrt{-1}:
1/6\leqslant x<1/3, 0\leqslant y<1\}.
\end{equation}

A collection of $N\geqslant 1$ points 
\[
b_1,\dots, b_N\in\mathcal{D}_{\mathsf{R}}
	:=
	\{x+y\sqrt{-1}:
	1/2\leqslant x<1, 0\leqslant y<1\},
\]
located in the right-half of $\mathcal{D}$,
 is said to be {\sl distributed sparsely}, if for any disc $\mathbb{D}(a,r)$, the following cardinality estimate holds
\begin{equation}
	\label{distribute sparsely}
	\big|\mathbb{D}(a,r)
	\cap \{b_1,\dots,b_N\}\big|
	\leqslant\max\{1, \K\cdot r^2 N\}.
\end{equation}
For  instance,  this can be reached by choosing distinct points
\[
b_1, \dots, b_N
\in
\Big\{
\frac{[\sqrt{N}]+1+\ell_1}{2[\sqrt{N}]+2}
+
\frac{\ell_2}{[\sqrt{N}]+1}\sqrt{-1}
\Big\}_{0\leqslant\ell_1,\ell_2\leqslant[\sqrt{N}]
}.
\]

Let $\mathcal{S}=\{I\subset \mathbb{Z}_{+}: I\text{ is a finite nonempty set, or } I=\varnothing, \text{ or } I=\mathbb{Z}_{+}\}$. Then $\mathcal{S}$ is  countable, i.e., there exists some  bijection  $\sigma:\mathcal{S}\rightarrow\mathbb{Z}_{+}$. On the other hand, we can decompose $\mathbb{Z}_{+}$ into some infinite disjoint union $\cup_{i\in\mathbb{Z}_{+}}\mathcal{Z}_i$, where each component $\mathcal{Z}_i$ contains infinitely many integers. For every $I\in\mathcal{S}$, write all the elements of $\mathcal{Z}_{\sigma(I)}$ in the increasing order as
$
Z_1^{I}<{Z}_2^{I}<{Z}_3^{I}<\cdots
$.
Thus we can rearrange
\begin{equation*}
	\label{decompostion of Z+}
	\mathbb{Z}_{+}
	=
	\cup_{I\in \mathcal{S}}\mathcal{Z}_{\sigma(I)}
	=\cup_{I\in\mathcal{S}}
	\cup_{j\geqslant 1}
	\{{Z}_j^{I}\}.
\end{equation*}

For every positive integer $i=Z^I_j$,
we now choose all the $x_{\mu}\in \mathcal{D}$ for  $\mu\in
 \A_{r_i}\cap c\Gamma$ as follows. 

\begin{itemize}
	\item{Case I}: $I=\varnothing$.
	
	We require that all the $x_{\mu}$'s 
	are distributed sparsely in $\mathcal{D}_{\mathsf{R}}$.
	
	\smallskip
	
	\item{Case II}: $I=\{i_1,\dots,i_k\}$ is some finite set of  $k\geqslant 1$ elements, and $j\geqslant 1$ is an odd integer. 
	
	We choose all
	$x_{\mu}$  from $\{y_{i_1}, \dots, y_{i_k}\}$, such that, for every
	$\ell=1,\dots, k$, $x_{\mu}=y_{i_\ell}$ for at least
	$[\frac{|\A_{r_i}\cap c\Gamma|}{k}]$ times.
	
	\smallskip
	
	\item{Case II'}: $I=\{i_1,\dots,i_k\}$ is some finite set of  $k\geqslant 1$ elements, and $j\geqslant 1$ is an even integer.

	We choose some $x_{\mu}=y_{i_\ell}$ for 
	$[\frac{|\A_{r_i}\cap c\Gamma|}{2k}]$ times, where $\ell=1,\dots, k$,
	and we require the  remaining $x_{\mu}$'s  to be distributed sparsely  in $\mathcal{D}_{\mathsf{R}}$. 
	
	\smallskip
	
	\item{Case III}: $I=\mathbb{Z}_{+}$, and $j\geqslant 1$ is odd. 
	
	Fix a sequence of positive numbers $\{\alpha_j\}_{j=1}^{\infty}$ with $\sum_{j=1}^{\infty}\alpha_j=1$. 
	We choose all
	$x_{\mu}$  from $\{y_{\ell}\}_{\ell=1}^{\infty}$, such that, for every
 $\ell\geqslant 1$,
 $x_{\mu}=y_{\ell}$ for at least  
	$[\alpha_{\ell}\cdot |\A_{r_i}\cap c\Gamma|]$ times.
	
	\smallskip
	
	\item{Case III'}   $I=\mathbb{Z}_{+}$, and $j\geqslant 1$ is even. 
	
	For every $\ell\geqslant 1$,
	we choose some $x_{\mu}=y_{\ell}$ for 
	$[\frac{\alpha_{\ell}}{2}\cdot|\A_{r_i}\cap c\Gamma|]$ times;
	and we choose the remaining $x_{\mu}$'s
	to be distributed sparsely  in $\mathcal{D}_{\mathsf{R}}$.
\end{itemize}

\section{\bf Proofs}\label{section: proof}
We are now in position to prove the main results. Recall that from a given sequence of discs of increasing radii $r_i\nearrow \infty$, after a perturbation and passing to some subsequence, we can always receive an Ahlfors current for $f$.

\begin{obs}
	\label{obs 6.1}
	From the sequence of radii $\{\frac{r_i}{3}\}_{i\geqslant 1}$, one receives a singular Ahlfors current  $T$ of the shape
	\[
	T=c_{\infty}\cdot[\mathcal{C}_{\infty}].
	\]
\end{obs}

\begin{proof}
	Note that
	all points in 
	$\D_{r_i/3}\setminus \D_{r_{i-1}+2}$
	keep positive distance
	$\geqslant 2-\sqrt{2}$ to $\Lambda$. Thus for any small open neighborhood $U$ of $\mathcal{C}_{\infty}$,
	by Subsection~\ref{f(z) near infinity curve}, for
	$i\gg 1$, for every
	$z\in\D_{r_i/3}\setminus \D_{r_{i-1}+2}$, we have $f(z)\in U$.
	Note that for $i\gg 1$, by~\eqref{disc image area is bounded} and~\eqref{disc area is large enough}, one see that the area of $f(\mathbb{D}_{r_{i-1}+2})$ is
	 negligible comparing with that of $f(\mathbb{D}_{r_{i}/3})$, namely
	\[
	\int_{\mathbb{D}_{r_{i-1}+2}}f^*\omega_X
	\leqslant
	\K\cdot r_{i-1}^2
	=
	o(1)
	\cdot
	r_i^2
	\leqslant
	o(1)
	\cdot
	\int_{\mathbb{D}_{r_{i}/3}}f^*\omega_X.
	\]
	Thus $T$  charges zero mass outside $U$. Since this holds true for any open neighborhood $U\supset \mathcal{C}_{\infty}$, we conclude that $T$ must be supported on $\mathcal{C}_{\infty}$.
\end{proof}

\begin{obs}
	\label{obs 6.2}
	From the sequence of radii $\{r_{
		Z^{\varnothing}_{j}}\}_{j\geqslant 1}$, one gets an Ahlfors current $T$ having the shape
	\[
	T=a_{\infty}\cdot[\mathcal{C}_{\infty}]+ T_{\diff},
	\]
	where $a_{\infty}$ is some positive number and $T_{\diff}$ is a nontrivial diffuse part.
\end{obs}

\begin{proof}
	{\sl Step 1}: $T$ charges positive mass along $\mathcal{C}_{\infty}$. 
	
	Indeed,
	for any open neighborhood $U$ of $\mathcal{C}_{\infty}$,
	it follows from the preceding proof and the estimate~\eqref{disc area is large enough} that, for $j\gg 1$ and for $i=Z^{\varnothing}_{j}$, we have
	$f(\D_{r_i/3}\setminus \D_{r_{i-1}+2})\subset U$
	 and
$
	\int_{
\D_{r_i/3}\setminus \D_{r_{i-1}+2}	
}f^*\omega_X
	\geqslant
	\K
	\cdot
	r_i^2$.
	On the other hand, by~\eqref{disc image area is bounded} we know that
	$\int_{
		\D_{r_i}	
	}f^*\omega_X
	\leqslant
	\K\cdot r_i^2$.
	Thus $T$ charges $U$ by positive mass $\geqslant \K>0$. Since this holds true for any $U$, we conclude that
	$T$ charges positive mass along $\mathcal{C}_{\infty}$.
	
	\smallskip
	
	{\sl Step 2}: $T$ does not charge any other algebraic curve.
	
	If $D\neq \mathcal{C}_{\infty}$ is an irreducible horizontal curve, using the same notations as that of Subsection~\ref{estimate area near horizontal curves},
	 by the
	 estimate~\eqref{very subtle estimate},
	 and by choosing $\epsilon'\leqslant \epsilon$, we know that
	 $T$ charges the neighborhood
	 $\vartheta^{-1}(U_1^{\epsilon})$ by a small mass $\leqslant \K_{c_0, \delta, \vartheta}\cdot {\epsilon}^2$.
	 Letting $\epsilon\rightarrow 0$, we receive that $T$ charges no mass on $U_1\subset D$. Thus $T$ cannot charge positive mass on $D$.
	 
	 If $D=\mathbb{P}^1_{a}$ is an irreducible vertical curve, for an open neighborhood $U$ of $\mathcal{C}_{\infty}$,
	 we claim that
	 $T$ charges no mass on $D\setminus U$. Indeed,
	 for any small $\epsilon>0$,
	 by Subsection~\ref{f(z) near infinity curve},
	 for $|z|\gg 1$ with
	 $f(z)\in \pi_{1}^{-1}(\D(a, \epsilon))\setminus U$, there must be some $\lambda\in\Lambda$ such that
	 $z\in \D(\lambda, \epsilon)$. Note also that $\pi_1(z)=[z]$, we get
	 $[\lambda]\in \D(a, 2\epsilon)$.
	  By~\eqref{distribute sparsely}, we have
	 \[
	 |
	 \pi_0^{-1}(\D(a, 2\epsilon))\cap \Lambda\cap \D_{r_{Z^{\varnothing}_j}}
	 |
	 \leqslant
	 \K
	  \epsilon^2
	 \cdot
	 r_{Z^{\varnothing}_j}^2.
	 \]
	 Hence by the estimates~\eqref{small disc high bound} and~\eqref{disc area is large enough},
	 such points $z\in \D_{r_{Z^{\varnothing}_j}}$  with
	 $f(z)\in \pi_{1}^{-1}(\D(a, \epsilon))\setminus U$ constitute 
	 only small portion of area measured by $f^*\omega_X$, thus
	  $T$ charges mass $\leqslant \K\cdot \epsilon^2$
	 over $\pi_{1}^{-1}(\D(a, \epsilon))\setminus U$. Letting $\epsilon\rightarrow 0$, we conclude the claim. Since this holds for any open neighborhood $U$, we receive that $T$  does not charge $\mathbb{P}_a^1$, which finishes this step.

		\smallskip
	
	{\sl Step 3}: $T$ has positive mass outside $\mathcal{C}_{\infty}$.
	
	Take any small open neighborhood $U$ of
	$\mathcal{C}_\infty$ such that $\overline{U}\cap \mathcal{C}_0$ is empty.
	Note that for every $\lambda\in \Lambda\cap \A_{r_{Z^{\varnothing}_j}}$ where $j\gg 1$, by~\eqref{very subtle estimate again}, the image of $f$ about $\lambda$
	contributes $\geqslant \K$  area outside $U$. 
	Moreover, by our construction
	$|\Lambda\cap \A_{r_{Z^{\varnothing}_j}}|\geqslant\K\cdot r_{Z^{\varnothing}_j}^2$,
	thus the total area of 
	$f(\D_{r_{Z^{\varnothing}_j}})\setminus U$ is
	$\geqslant \K\cdot  r_{Z^{\varnothing}_j}^2$. Lastly, by~\eqref{disc image area is bounded}, we conclude that
	$T$ charges positive mass outside $U$.
\end{proof}

\begin{obs}\label{obs 6.3}
	Fix any finite subset $I=\{i_1,\dots,i_k\}\subset \mathbb{Z}_+$ having cardinality  $k\geqslant 1$.
	Then from the sequence of radii $\{r_{Z^{I}_{2j-1}}\}_{
		j\geqslant 1}$, one receives an Ahlfors current $T$
	having the shape
	\[
	T=a_{\infty}\cdot[\mathcal{C}_{\infty}]+
	\sum_{\ell=1}^k
	a_{i_{\ell}}\cdot
	[\mathbb{P}^1_{[y_{i_{\ell}}]}],
	\]
	where $a_{\infty}, a_{i_1}, \dots, a_{i_k}$ are some positive numbers.
\end{obs}

\begin{proof}
	For any small open neighborhood $U$ of $\mathcal{C}_{\infty}\cup \mathbb{P}^1_{[y_{i_{1}}]}
	\cup
	\cdots
	\cup \mathbb{P}^1_{[y_{i_{k}}]}$,
	by Subsection~\ref{f(z) near infinity curve}, for
	$j\gg 1$,
	for all points
	$z\in\D_{r_{Z^{I}_{2j-1}}}\setminus \D_{r_{Z^{I}_{2j-1}-1}+2}$
	we have $f(z)\in U$.
	Indeed, choose a very small  $\epsilon>0$ such that
	$U$ contains $\pi_1^{-1}(\D([y_{i_\ell}], \epsilon))$ for every $\ell=1, \dots, k$. Then if $\text{dist}(z, \Lambda)\geqslant \epsilon$, we know that $f(z)$ is very close to $\mathcal{C}_{\infty}$, whence $f(z)\in U$; 
	otherwise $\text{dist}(z, \Lambda)< \epsilon$,
	that is $z\in \D(\lambda, \epsilon)$
	for some $\lambda\in \pi_0^{-1}([y_{i_\ell}])$ ($\ell=1, \dots, k$) by our construction of $\Lambda$, hence $f(z)\in \pi_1^{-1}(\D([y_{i_\ell}], \epsilon))\subset U$. 
	
	Therefore, by the same argument as that of Observation~\ref{obs 6.1}, one sees that $T$ is supported in
	$\mathcal{C}_{\infty}\cup \mathbb{P}^1_{[y_{i_{1}}]}
	\cup
	\cdots
	\cup \mathbb{P}^1_{[y_{i_{k}}]}$. It remains to  check that $T$ charges positive mass in each of these components.
	
	Indeed, first of all,  our algorithm guarantees that
	\[
	|\Lambda\cap \A_{r_{Z^{I}_{2j-1}}}
\cap
\pi_0^{-1}([y_{i_\ell}])
|\geqslant\K\cdot r_{Z^{I}_{2j-1}}^2
\qquad(\ell=1, \dots, k).
	\]
	By~\eqref{small disc low bound}, for any fixed small  $\epsilon>0$, for large $j\gg 1$ and $i=r_{Z^{I}_{2j-1}}$,
	for any $\lambda\in \Lambda\cap \A_{r_{i}}
	\cap
	\pi_0^{-1}([y_{i_\ell}])$, the holomorphic disc $f(\mathbb{D}(\lambda,\epsilon))$
	is contained in $\pi_1^{-1}(\D([y_{i_\ell}], \epsilon))$ with  area
	$\int_{\mathbb{D}(\lambda,\epsilon)}
	f^*
	\omega_X\geqslant \K$  bounded from below by some uniformly positively constant independent of $\epsilon$.
	Thus the total area  of such discs is $\geqslant \K\cdot r_i^2$. Noting that ~\eqref{disc image area is bounded} implies $\int_{\D(r_i)}\leqslant \K \cdot r_i^2$,
	thus $T$ charges mass $\geqslant \K$ on $\pi_1^{-1}(\D([y_{i_\ell}], \epsilon))$.
	Letting $\epsilon\rightarrow 0$, we conclude that $T$ charges positive mass on $\mathbb{P}_{[y_{i_\ell}]}^1$. Lastly, by the same argument as the Step 1 of Observation~\ref{obs 6.2}, we see that
	$T$ charges positive mass on $\mathcal{C}_{\infty}$.
Thus we conclude the proof.
\end{proof}

\begin{obs}
	Fix any finite subset $I=\{i_1,\dots,i_k\}\subset \mathbb{Z}_+$ having cardinality  $k\geqslant 1$.
	Then from the sequence of radii $\{r_{Z^{I}_{2j}}\}_{
		j\geqslant 1}$, one receives an Ahlfors current $T$
	having the shape 
	\[
	T=a_{\infty}\cdot[\mathcal{C}_{\infty}]+
	\sum_{\ell=1}^k
	a_{i_{\ell}}\cdot
	[\mathbb{P}^1_{[y_{i_{\ell}}]}]
	+
	T_{\diff},
	\]
	where $a_{\infty}$, $a_{i_{\ell}}$ ($1\leqslant \ell\leqslant k$) are some positive constants and where $T_{\diff}$ is a nontrivial diffuse part.
\end{obs}

\begin{proof}
	{\sl Step 1}: $T$ charges positive mass along $\mathcal{C}_{\infty},  \mathbb{P}^1_{[y_{i_{1}}]}, \cdots,  \mathbb{P}^1_{[y_{i_{k}}]}$.
	
	This follows from the same arguments as in the preceding proof.
	
	\smallskip
	
	{\sl Step 2}: $T$ does not charge any other algebraic curve. 
	
We can check it by using the same arguments as in the Step 2 of Observation~\ref{obs 6.2}.
	
	\smallskip
	
	{\sl Step 3}: $T$ has positive mass outside $\mathcal{C}_{\infty}\cup \mathbb{P}^1_{[y_{i_{1}}]}
	\cup
	\cdots
	\cup \mathbb{P}^1_{[y_{i_{k}}]}$. 
	
	The argument is similar to the Step 3 of Observation~\ref{obs 6.2}. The key point is that, by our algorithm,
	 \[
	|
	(\mathcal{D}_{\mathsf{R}}+\Gamma)\cap \Lambda\cap \D_{r_{Z^{I}_{2j}}}
	|
	\geqslant
	\K
	\cdot
	r_{Z^{I}_{2j}}^2,
	\]
	and for every  $\lambda\in (\mathcal{D}_{\mathsf{R}}+\Gamma)\cap \Lambda\cap \D_{r_{Z^{I}_{2j}}}$,
	$[\lambda]$ keeps uniform distances $\geqslant \epsilon>0$ to $[y_{i_1}], \dots, [y_{i_k}]$.
	Assuming moreover that $j\gg 1$,  in the same notation as Observation~\ref{obs 4.2} we receive
	\[
		\area\Big(
		(X\setminus U)
		\cap
		f\big(
		\D(\lambda, \epsilon/2)
		\big)
		\Big)_{\omega_{X}}
		\geqslant
		\K_U.
\]
Note that 
$
f\big(
\D(\lambda, \epsilon/2)
\big)$
stays away from
$\pi_1^{-1}(\D([y_{i_\ell}], \epsilon/2))$ for all $\ell=1, \dots, k$. Thus by the same argument as the preceding proof, 
we see that $T$ charges positive mass outside $U\cup \big(\cup_{\ell=1}^k \pi_1^{-1}(\D([y_{i_\ell}], \epsilon/2))\big)$.
\end{proof}

By much the same arguments, we have the following two results.

\begin{obs}
	From the sequence of radii $\{r_{Z^{\mathbb{Z}_{+}}_{2j-1}}\}_{
		j\geqslant 1}$, one receives an Ahlfors current $T$
	having the shape 
	\[
	T=a_{\infty}\cdot[\mathcal{C}_{\infty}]+
	\sum_{\ell= 1}^{\infty}
	a_{{\ell}}\cdot
	[\mathbb{P}^1_{[y_{{\ell}}]}],
	\]
	where $a_{\infty}$, $a_{{\ell}}$ ($\ell\geqslant 1$) are  positive numbers.
	\qed
\end{obs}

\begin{obs}
	\label{obs 6.6}
	From the sequence of radii $\{r_{Z^{\mathbb{Z}_{+}}_{2j}}\}_{
		j\geqslant 1}$, one receives an Ahlfors current $T$
	having the shape 
	\[
	T=a_{\infty}\cdot[\mathcal{C}_{\infty}]+
	\sum_{\ell= 1}^{\infty}
	a_{{\ell}}\cdot
	[\mathbb{P}^1_{[y_{{\ell}}]}]
	+
T_{\diff},
	\]
	where $a_{\infty}$, $a_{{\ell}}$ ($\ell\geqslant 1$) are  positive numbers,
	and where $T_{\diff}$ is a nontrivial diffuse part.
	\qed
\end{obs}

Thus we prove Theorems~\ref{thm 1},~\ref{thm 2}.

\section{\bf Examples}
\label{section: examples}
\subsection{Diffuse Ahlfors currents}

Let $\mathcal{A}=\mathbb{C}/\Lambda\times\mathbb{C}/\Lambda$ be the surface obtained as the product of two elliptic curves where $\Lambda$ is a lattice. 
Fix a reference metric $\omega_{\mathcal{A}}:=
\dif \dif^c |z_1|^2
+
\dif \dif^c |z_2|^2$ on ${\mathcal{A}}$.
 Choose an irrational number $\lambda\in\mathbb{R}\setminus\mathbb{Q}$. Consider the holomorphic curve $f:\mathbb{C}\longrightarrow\mathcal{A}$ given by $f(z)=([z],[\lambda z])$. 

\begin{pro}
	\label{diffuse in Abelian surface}
	Any Ahlfors current $T$ of $f$ is diffuse.
\end{pro}

\begin{proof}
Since there is no nonconstant holomorphic map from $\mathbb{P}^1(\mathbb{C})$ to an elliptic curve,  $\mathcal{A}$ contains no rational curve.
Hence
by a theorem of Duval \cite{Duval2006}, it suffices to check that $T$ charges zero mass along any elliptic curve in $\mathcal{A}$.

\smallskip
\noindent
{\bf Fact}
(c.f. \cite[Prop.~1.3.2]{Diamond-Shurman2005}).
 {\it
Let $\Phi 
:
\mathbb{C}/\Gamma_1 
\longrightarrow
 \mathbb{C}/\Gamma_2
$
be a holomorphic map between
complex tori. Then there exist complex numbers $m, b$ with
$m\Gamma_1\subset \Gamma_2$,
such that
$\Phi([z]) 
= [mz+b]$.
} 

\smallskip
Therefore, any nonconstant holomorphic map  $\iota: \mathbb{C}/\Gamma_3\rightarrow \mathbb{C}/\Lambda\times\mathbb{C}/\Lambda$ from an elliptic curve
to $\mathcal{A}$ can be written explicitly as
$\iota([z])=([m_1 z+b_1], [m_2 z+ b_2])$
for some complex numbers $m_1, m_2, b_1, b_2$, such that $(m_1, m_2)\neq (0, 0)$ and $m_1\Gamma_3, m_2\Gamma_3\subset \Gamma$. Hence $m_2-\lambda m_1\neq 0$.
We claim that the  intersection numbers
\begin{equation}
	\label{counting intersection nubmers on tori}
|\iota(\mathbb{C}/\Gamma_3)\cap f(\D_{r})|\leqslant \K\cdot r^2
\end{equation}
for $r\gg 1$.
Indeed, we can find
a large disc $\D_R$ containing a fundamental domain of $\Gamma_3$.
Then for $z\in \D_r, y\in \D_R$ with $([z], [\lambda z])=([m_1 y+b_1], [m_2 y+ b_2])$,
we receive that 
\begin{equation}
	\label{solve the linear equation}
z-(m_1 y+b_1)=\lambda_1, \qquad
\lambda z-(m_2 y+ b_2)=\lambda_2
\end{equation}
 for some $\lambda_1, \lambda_2\in\Lambda$ having absolute values less than $ r+\K, |\lambda|\cdot r+\K$ respectively.
 
 Since $m_2-\lambda m_1\neq 0$, we can solve the linear equation~\eqref{solve the linear equation} as
 \[
 z=\frac{m_2(\lambda_1+b_1)-m_1(\lambda_2+b_2)}{m_2-\lambda m_1},
 \qquad
 y=
 \frac{\lambda(\lambda_1+b_1)-(\lambda_2+b_2)}{m_2-\lambda m_1}.
 \]
 Noting that $y\in \D_R$,
 for any fixed $\lambda_1$, the cardinality of possible choices
 of 
 \[
 \lambda_2
 \in
 \big(
 (-m_2+\lambda m_1)\cdot 
 \D_R
 +
 \lambda(\lambda_1+b_1)-b_2
 \big)
 \cap \Lambda
 \] is 
 $\leqslant \K$.
Thus the cardinality of possible choices of such $(\lambda_1, \lambda_2)\in \Lambda\times \Lambda$
is $\leqslant \K\cdot (r+\K)^2\cdot \K\leqslant \K \cdot r^2$. Hence the estimate~\eqref{counting intersection nubmers on tori} is proved.

By the compactness of $\iota(\mathbb{C}/\Gamma_3)$, and by shrinking neighborhood $U$ of $\iota(\mathbb{C}/\Gamma_3)$ if necessary, each intersection point  corresponds to a small area $o(1)$ component of $f(\D_r)\cap U$, thus
the total area of $f(\D_r)\cap U$
is $\leqslant o(1)\K\cdot r^2$.
However, the area growth of $f(\D_r)$ is $\K_{\lambda}\cdot r^2$. Thus any obtained Ahlfors current of $f$ charges  mass
$\leqslant  o(1)\K$ on $U$. By shrinking $U$, we know that $T$ charges zero mass
 along $\iota(\mathbb{C}/\Gamma_3)$.
 Hence
 we conclude the proof.
 \end{proof}

	Take a holomorphic surjective map $\pi_2: \mathcal{A}\longrightarrow \mathbb{P}^2(\mathbb{C})$, which induces an entire curve
	\[
	f_2:=\pi_2\circ f:\,\, \mathbb{C}\,\,\longrightarrow\,\, \mathbb{P}^2(\mathbb{C}).
	\]
		Since $\pi_2^*\omega_{\FS}\geqslant 0$ is closed,
	by the geometry of $\mathcal{A}$, in the cohomology class $[\pi_2^*\omega_{\FS}]$
	we can find a harmonic representative
	\[
	\omega=a_1 \sqrt{-1}\dif z_1\wedge \dif \overline{z_1}
	+
	a_2 \sqrt{-1}
	\dif z_2\wedge \dif \overline{z_2}
	+
	a_3 \sqrt{-1}\dif z_1\wedge \dif \overline{z_2}
	+
	a_4 \sqrt{-1}\dif z_2\wedge \dif \overline{z_1}
	\geqslant 0
	\]
	for some constants  $a_1, a_2, a_3, a_4$. Thus
	$f^*\omega=K \sqrt{-1}\dif z\wedge\dif \overline{ z}
	\geqslant 0$ where
	\[
	K=a_1+a_2 \lambda^2+\lambda(a_3+a_4)\geqslant 0.
	\]
	Since 
	$a_1, a_2, a_3+a_4$ cannot vanish simultaneously,
	 at most one $\lambda$ in
	$ \mathbb{R}\setminus \mathbb{Q}$ can make $K=0$. Now we only choose  $\lambda\in \mathbb{R}\setminus \mathbb{Q}$
	such that $K>0$.
	
	\begin{pro}
		\label{diffuse Ahlfors currents in Cp2}
	Any Ahlfors current $T_2$ of $f_2$ is diffuse.
		\end{pro}
	
	It is interesting to see that $f_2$ is tangent to a multi-valued vector field
	induced by the push-forward of the constant vector field $(1, \lambda)$ on $\mathcal{A}$.
	
	\begin{proof}
	Assume that $T_2$ is obtained by an increasing radii $\{r_i\}_{i\geqslant 1} \nearrow \infty$.
	By our chosen metric $\omega_{\mathcal{A}}$, the ``Length-Area'' condition of Ahlfors' lemma is automatically satisfied, thus by  passing to some subsequence
	$\{r_{i_k}\}_{k\geqslant 1}$
	we can receive an Ahlfors current $T$ of $f$.

Noting that $f^*\omega=K \sqrt{-1}\dif z\wedge\dif \overline{ z}
> 0$,
	by closeness of $T$
	and by $\area(f(\D_{r}))_{\omega_{\mathcal{A}}}=\K\cdot r^2$, we  receive 
	 \[
	T(\pi_2^*\omega_{\FS})=
	T(\omega)\geqslant \K
	>0.
	\]
	By the construction of $T$,
we receive that
	\[
	\area\big(f_2(\D_{r_{i_k}})\big)_{\omega_{\FS}}\geqslant \K\cdot r_{i_k}^2 \qquad
	{\scriptstyle(k\,\gg\, 1)}.
	\]
	
	Fix some
	$\K$ such that $\pi_2^*\omega_{\FS}\leqslant \K\cdot \omega_{\mathcal{A}}$. For any irreducible curve $C\subset \mathbb{P}^2(\mathbb{C})$, for any open neighborhood $U$ of $C$,
	we have
	\[
	\area\big(f_2(\D_{{r_{i_k}}})\cap U\big)_{\omega_{\FS}}
	=
	\area\big(f(\D_{r_{i_k}})\cap \pi_2^{-1}(U)\big)_{\pi_2^*\omega_{\FS}}
	\leqslant 
	\K\cdot
	\area\big(f(\D_{r_{r_{i_k}}})\cap \pi_2^{-1}(U)\big)_{\omega_{\mathcal{A}}}.
	\]
	Since $T$ charges no mass along $\pi_2^{-1}(C)$ by Proposition~\ref{diffuse in Abelian surface},  by shrinking $U$,
	the right-hand-side above  is
	$\leqslant 
	o(1)\cdot r_{i_k}^2$.
	Thus $T_2$ charges no mass along $C$. Since $C$ is arbitrary, 
	we conclude the proof.
	\end{proof}

\subsection{Singular Nevanlinna currents on $X$}
\label{Singular Nevanlinna currents on X}
Replacing ``Ahlfors currents'' by ``Nevanlinna currents'' in Observations~\ref{obs 6.1} -- \ref{obs 6.6},
the same statements still hold true by much the same arguments. Indeed, every upper or lower bound about  $\int_{\mathbb{D}(\lambda,\epsilon)}
f^*
\omega_X$ or $\int_{\D_{r_i}}
f^*
\omega_X$ has a corresponding one
about order function.
A remaining technical detail we would like to mention is the following

\begin{obs}
	For every $\ell\geqslant 1$,
	there exists some positive $\beta_\ell<1$ 
	such that, for $j\gg 1$ and $i=Z^{\mathbb{Z}_+}_j$, 
	\[
	|\Lambda\cap \A_{r_i}
	\cap \D_{\beta_\ell r_i}\cap \pi_0^{-1}([y_{\ell}])
	|
	\geqslant
	\K_{\ell}
	\cdot
	r_i^2.
	\]
\end{obs}

It will be helpful to show that certain  Nevanlinna currents of $f$ charge positive mass along $\mathbb{P}^1_{[y_\ell]}$.

\begin{proof}
	Note that for $j\gg 1$ we have
	$
	|\Lambda\cap \A_{r_i}
	\cap  \pi_0^{-1}([y_{\ell}])
	|
	>
	\frac{\alpha_{\ell}}{3}
	\cdot 
	\K
	\cdot
	r_i^2
	$. Moreover, for any fixed $\beta<1$,
	for $j\gg 1$, we have
	$
	|\Lambda\cap \A_{r_i}
	\setminus
	\D_{\beta r_i})|
	\leqslant \K\cdot (1-\beta) r_i^2
	$. By these two estimates, we can conclude the proof.
	\end{proof}

Therefore we can replace  ``Ahlfors currents'' by ``Nevanlinna currents'' in the statements of Theorems~\ref{thm 1},~\ref{thm 2}. Also, by much the same proofs,
Propositions~\ref{diffuse in Abelian surface}, \ref{diffuse Ahlfors currents in Cp2} also hold true for Nevanlinna currents.

\subsection{Singular Ahlfors currents on blow-ups of $X$}
We sketch a  construction of elaborate (in the sense of cohomology classes) singular Ahlfors currents,
 on the blow-ups of $X$ having Picard numbers $\geqslant 3$.

For any given positive integer $n\geqslant 1$, recall the collection of points  $y_1, \dots, y_n$ given in~\eqref{choose y_i}, let
 $\mathcal{X}$ be
the blow-up of $X$ at these points with the corresponding exceptional divisors  $E_1, \dots, E_n$.
Let $\mathsf{p}: \mathcal{X}\rightarrow X$ be the projection.
We now use the section $\psi^2\cdot s_m$ instead of $\psi\cdot s_m$
to induce an entire curve $\mathsf{f}:\mathbb{C}\longrightarrow X$. By lifting we thus receive an entire curve $\zeta: \mathbb{C}\longrightarrow \mathcal{X}$. We   strengthen our choices of $m, c$ in~\eqref{key choices of m, k}
by the condition $m\cdot\alpha-2\K/c^2
>0$, to make sure that the same clustering phenomenon as in Subsection~\ref{f(z) near infinity curve} 
 holds true
for $\mathsf{f}$
and
$\mathcal{C}_{\infty}$.
Let $e_i$ be the intersection point of the strict transformation $\widetilde{\mathcal{C}_0}$ of $\mathcal{C}_0$ with $E_i$ ($i=1,\dots, n$).
The purpose of using $\psi^2$ instead of $\psi$ is to make sure that, for $\lambda\in \Lambda$ with $[\lambda]=[y_i]$,
we have the certain value $\zeta(\lambda)=e_i$.

It is well-known that,
there exist some  hermitian metrics $h_{i}$ of the line bundles $\mathcal{O}(-E_i)$ and some small positive constant $\epsilon_2\ll 1$
such that
$\omega_{\mathcal{X}}:=\mathsf{p}^*\omega_X+\epsilon_2\sum_{\ell=1}^{n}\Theta_{h_{\ell}
}
$
is a K\"ahler form on $\mathcal{X}$
(c.f. \cite[Proposition 3.24]{Voisin2007-I}).
Moreover,
comparing the lifting
$\zeta: \mathbb{C}\longrightarrow \mathcal{X}$ with $\mathsf{f}: \mathbb{C}\longrightarrow X$,
we have
\[
T_{\zeta,r}(\omega_{\mathcal{X}})
:=
\int_{1}^r\frac{\dif t}{t}\int_{\D_t}
\zeta^*\omega_{\mathcal{X}}
\leqslant 
T_{\mathsf{f},r}(\omega_X)+O(1),
\]
(c.f. \cite[page~64, Observation 2.5.1]{Huynh2016}).
Thus we can use the same arguments for~\eqref{disc image area is bounded} to conclude that
\[
\int_{\D_{2r_i}}{\zeta}^*\omega_{\mathcal{X}}\leqslant\K\cdot r_i^2.
\]

For $\lambda\in \Lambda$ with $[\lambda]=[y_i]$,
computing in local coordinates around $e_i$,
for any small  $\epsilon>0$,
for any open neighborhood $U$ of $E_i$,
assuming further that $|\lambda|\gg 1$, then the area of $\zeta(\D(\lambda, \epsilon))\cap U$
is uniformly positively bounded (independent of $U$ and $\epsilon$) from below by using Propositions ~\ref{the most difficult estimate of psi} and~\ref{curve in ball area}.

Therefore, as an analogue of Observation~\ref{obs 6.3},
	for the finite subset $I=\{1, \dots, n\}$,
 from the sequence of radii $\{r_{Z^{I}_{2j-1}}\}_{
	j\geqslant 1}$, after a perturbation and passing to a subsequence, we can receive an Ahlfors current 
\[
T=a_{\infty}\cdot[
\widetilde{\mathcal{C}_{\infty}}]+
\sum_{\ell=1}^n
a_{{\ell}}\cdot
[\widetilde{\mathbb{P}^1_{[y_{\ell}]}}]
+
\sum_{\ell=1}^n
b_{{\ell}}\cdot
[E_i],
\]
where $\widetilde{\mathcal{C}_{\infty}}$, $\widetilde{\mathbb{P}^1_{[y_{\ell}]}}$  stand for the strict transformations of $\mathcal{C}_{\infty}$, $\mathbb{P}^1_{[y_{\ell}]}$, and where $a_{\infty}, a_{\ell}, b_{\ell}>0$ ($\ell=1, \dots, n$) are some positive numbers.

Similarly, we have the counterparts of  other Observations~\ref{obs 6.1} -- \ref{obs 6.6}.

\begin{center}
	\bibliographystyle{alpha}
	\bibliography{article}
\end{center}

\end{document}